\documentclass{amsart}
\usepackage{graphicx, amsfonts, comment}
\usepackage[dvipsnames]{xcolor}
\usepackage[margin=3cm, marginpar=2.5cm]{geometry}
\usepackage[colorlinks=true, urlcolor=NavyBlue, linkcolor=NavyBlue, citecolor=NavyBlue]{hyperref}

\renewcommand{\P}{\mathbb{P}}
\newcommand{\R}{\mathbb{R}}

\newcommand{\rptwo}{\R\P^2}
\newcommand{\Lambdatil}{\widetilde{\Lambda}}

\theoremstyle{plain}
\newtheorem{theorem}{Theorem}[section]

\newtheorem{lemma}[theorem]{Lemma}
\newtheorem{prop}[theorem]{Proposition}

\newtheorem{corollary}[theorem]{Corollary}
\newtheorem{proposition}[theorem]{Proposition}
\newtheorem{conjecture}[theorem]{Conjecture}

\newtheorem*{theorem*}{Theorem}
\newtheorem*{proposition*}{Proposition}

\theoremstyle{definition}
\newtheorem{question}[theorem]{Question}
\newtheorem{definition}[theorem]{Definition}

\theoremstyle{remark}
\newtheorem{remark}[theorem]{Remark}

\numberwithin{equation}{section}

\title{Multisections with divides and Weinstein 4-manifolds}
\author{Gabriel Islambouli}
\author{Laura Starkston }
\thanks{LS was supported by NSF 1904074, NSF 2042345, and a Sloan research fellowship.}

\begin{document}

\maketitle

\section{Introduction}

There are various methods to diagrammatically encode a $4$-dimensional manifold, each of which is based on a decomposition theorem which breaks up the manifold into simple pieces such that the diagram encodes the way these pieces glue together (e.g. handle decompositions, Lefschetz pencils/fibrations). The most recent such method is the development of trisections by Gay and Kirby~\cite{GayKir16}, generalized to multisections in~\cite{IslNay20}. Symplectic structures have played a key role in 4-dimensional topology, due to connections with gauge theory. Compatibility between symplectic topology with handle decompositions arose from Weinstein's construction~\cite{Wei91} and in the $4$-dimensional case was diagrammatically encoded by a Legendrian surgery diagram by Gompf~\cite{Gompf98}. Similarly compatibility between symplectic manifolds and Lefschetz pencils and fibrations was established~\cite{Don99,GompfStipsicz,LoiPie01}, so a symplectic manifold can be encoded by the fiber and base surfaces, pencil points, and ordered vanishing cycles. A notion of compatibility between trisections and symplectic manifolds was proposed in~\cite{LamMei18} and shown to exist in~\cite{LamMeiSta20}, but this compatibility did not yield a simple diagrammatic way to encode a symplectic structure (rather it was motivated by attempts to obtain genus bounds).

In this article we define a stronger compatibility between a multisection and a symplectic structure, which can be diagrammatically encoded by collections of curves on a surface. In addition to the diagrammatic data of the smooth multisection, we keep track of another multi-curve representing the \emph{dividing set} of convex surfaces in contact manifolds. Thus, we call our decomposition of a symplectic manifold a \emph{multisection with divides}. Our main result is that every $4$-dimensional Weinstein domain admits a multisection with divides.


\begin{definition}
\label{def:WeinsteinMultisection}
A  \textbf{multisection with divides} of a symplectic filling $(W,\omega)$ with contact boundary $(\partial W,\xi)$ is a decomposition $W=W_1\cup\cdots \cup W_n$, such that
\begin{itemize}
    \item $W_i\cong \natural_{k_i} S^1\times D^3$. 
    \item $W_i\cap W_{i+1}=:H_{i+1} \cong \natural_g S^1\times D^2$ for $i=1, \dots, n-1$.
    \item $\Sigma := W_1\cap \cdots \cap W_n = \partial H_i$ for all $i$
    \item Each $(W_i,\omega|_{W_i})$ is a symplectic filling of $(\partial W_i, \xi_i)$.
    \item $H_i\cup H_{i+1}$ is a contact Heegaard splitting of $(\partial W_i,\xi_i)$
    \item $H_1\cup H_{n+1}$ is a contact Heegaard splitting of $(\partial W,\xi)$
\end{itemize}
A \emph{bisection with divides} is a multisection with divides with $n=2$.
\end{definition}

    The advantage of using contact Heegaard splittings is that the handlebodies each carry a standard positive \emph{and} a standard negative contact structure, which are contactomorphic. This is one of the key ideas which we use to make multisections compatible with symplectic and contact geometry.

\begin{remark}
    In the fourth bullet point of Definition~\ref{def:WeinsteinMultisection}, we ask for $(W_i,\omega|_{W_i})$ to be a symplectic filling. Note that in our setting weak, strong, Liouville, and Weinstein fillability are all equivalent, because there is a unique weak symplectic filling of $\#_{k_i} S^1\times S^2$ up to symplectic deformation, and this filling is actually Weinstein (thus strong and Liouville)~\cite{NW}. 
\end{remark}

\begin{figure}
    \centering
    \includegraphics[scale=.25]{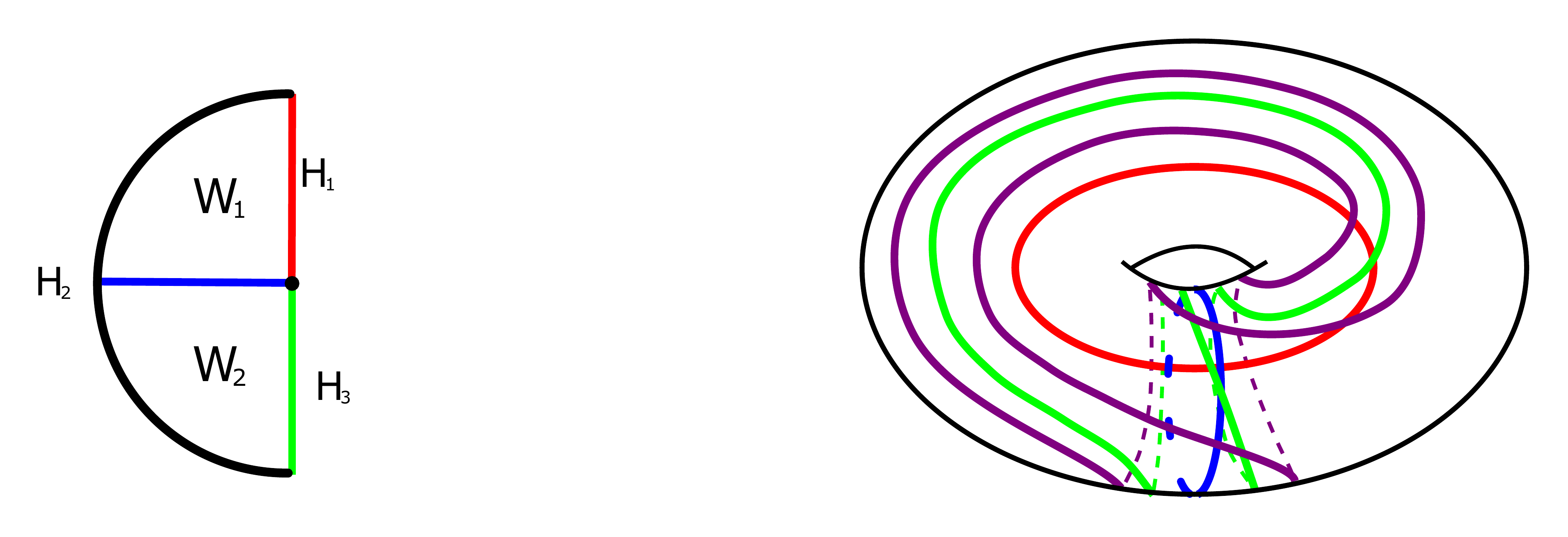}
    \caption{Left: A schematic of a bisection with divides. Each $W_i$ is a 4-dimensional Weinstein 1-handlebody and each $H_i$ is a 3-dimensional 1-handlebody obtained as a neighbourhood of a Legendrian graph. Right: A bisection diagram with divides of the unit cotangent bundle on $S^2$. The red, blue, and green curves represent curves bounding compressing disks in the respective handlebodies and the purple curves are the dividing set for the surface.}
    \label{fig:SchematicAndExample}
\end{figure}

An essential feature of these multisections with divides is that they can be encoded as a sequence of cut systems together with a fixed dividing set on a closed surface. An example of such a diagrammatic representation together with a schematic of what this encodes can be seen in Figure \ref{fig:SchematicAndExample}.

We are able to encode symplectic geometric data diagrammatically because our multisection with divides decompositions are geometrically restrictive by asking each Heegaard splitting to be a \emph{contact Heegaard splitting} (see section~\ref{ss:contactHS} for the definition). A typical multisection of a symplectic manifold would be unlikely to satisfy this condition, even if it were a ``Weinstein multisection'' as in~\cite{LamMei18,LamMeiSta20}. Therefore, it is surprising that these geometrically restrictive multisection decompositions actually exist quite generally. Our main theorem is the following.

\begin{theorem*}
    Every compact 4-dimensional Weinstein domain admits a multisection with divides.
\end{theorem*}

We give two proofs of this theorem each with distinct advantages. Both proofs also yield algorithmic methods to produce a diagram for the multisection with divides. The first proof takes as input a Kirby-Weinstein diagram, and produces a bisection with divides. The disadvantage of this algorithm is that the core surface will generally have very high genus, which is typically highly inefficient. On the other hand, the output only has two sectors, instead of arbitrarily many. 

The second proof takes as input a positive allowable Lefschetz fibration (PALF) and produces a multisection with divides. In this case, the genus is more controlled, being determined by the topology of the fiber of the PALF, however there may be many sectors (potentially one for each Lefschetz singularity). More specifically we prove the following.

\begin{theorem*}
    Let $f: W^4 \to D^2$ be a PALF whose regular fiber is a genus $g$ surface with $b$ boundary components and $n$ singular fibers. Then $W^4$ admits a genus $2g+b-1$ $n$-section with divides. 
\end{theorem*}

One can compare these results and Definition~\ref{def:WeinsteinMultisection} to the definition of Weinstein trisection for closed symplectic manifolds in~\cite{LamMei18,LamMeiSta20}. The main difference is that those prior definitions do not require any compatibility between the contact structure induced on the boundary of each sector and the Heegaard splitting of the boundary induced by the trisection. As a consequence, there is not easy diagrammatic data that encodes the contact and symplectic topology in these prior definitions. (The most likely candidate for such diagrammatic data is a weighted foliation for each handlebody, but the data of a weighted foliation is not discrete.) By contrast, in our more restrictive notion of multisection with divides, the symplectic and contact geometry can be diagrammatically encoded by a single dividing set on the core surface.

\begin{remark}
    Smooth multisections are compatible with both closed manifolds and manifolds with boundary. In this article, we have given the definition of a multisection for divides in the case that our symplectic manifold has contact boundary. The definition naturally extends to closed manifolds. Although in this article we establish the existence of our decompositions for Weinstein domains rather than closed symplectic manifolds, this is a key step towards establishing existence of the analogous multisections with divides for all closed symplectic manifolds via results of Donaldson~\cite{Don96} and Giroux~\cite{Gir02,Gir17}.
\end{remark}

The monodromy of a Lefschetz fibration is a product of right handed Dehn twists. In general, there can be multiple ways to write the same mapping class element as a product of right handed Dehn twists. Swapping out one of these with another is called a \emph{monodromy substitution}. A number of important symplectic cut and paste operations like rational blow-down can be seen as a monodromy substitution operation on a Lefschetz fibration \cite{EndGur10} \cite{EMV11}. By tracking the change induced by a monodromy substitution on a PALF through our algorithm, we are able to realize these cut and paste operations on multisections with divides. In Figure \ref{fig:lanternRelationOnBisections}, we demonstrate this explicitly for the monodromy substitution coming from the lantern relation, which induces a $C_2$-rational blowdown.

We conclude the paper with a classification of genus-1 multisections with divides. The smooth genus-1 multisections with boundary were previously classified in \cite{IKLM} to exist if and only if the manifold is a linear plumbings of 2-spheres. Our requirement that the multisections be compatible with genus-1 contact Heegaard splittings restricts this significantly more, as in the following theorem.

\begin{theorem*}
    Genus-1 $n$-sections with divides correspond to plumbings of $n-1$ disk bundles over 2-spheres, each of Euler number $-2$.
\end{theorem*}

The organization of this paper is as follows. Section~\ref{s:contacthandlebody} discusses the way we will ask contact structures to be compatible with Heegaard splittings and $3$-dimensional handlebodies. Section~\ref{s:KWdiag} gives our first proof of our main theorem, showing how to turn a Kirby-Weinstein handlebody diagram into a multisection with divides. Section~\ref{s:PALF} gives the second proof of our main theorem, showing how to turn a PALF into a multisection with divides. Section~\ref{s:g1stab} classifies genus-1 multisections with divides and shows how multisection diagrams with divides can be stabilized to increase the genus of the surface. Finally, in Section~\ref{s:questions} we discuss some questions for future research.

\section{Contact geometry and Heegaard splittings} \label{s:contacthandlebody}

In this section, we explain the compatibility condition between contact structures and Heegaard splittings which we will require on the boundary of each sector. We also give a diagrammatic formulation of this compatibility. We begin with some background on surfaces in contact manifolds.

A surface $\Sigma$ embedded in a contact 3-manifold $(Y, \xi)$ is said to be \textbf{convex} if there exists a contact vector field $v$ for $(Y, \xi)$ such that $v$ is transverse to $\Sigma$. Convex surfaces are generic, meaning every smoothly embedded surface has a $C^\infty$-small isotopy to a convex surface~\cite{Gir91,Hon00}. Given a contact vector field $v$ transverse to $\Sigma$ we obtain a multicurve called the \textbf{dividing set}, denoted $\Gamma_\Sigma$. This multicurve is defined by $\Gamma_\Sigma = \{x\in \Sigma | v_x \subset \xi \}$ and the isotopy class of this curve is an invariant of the embedding of $\Sigma$ up to isotopy through convex surfaces.

The dividing set captures all of the contact geometric information of $\Sigma$ in a neighbourhood $\Sigma \times [-\epsilon, \epsilon]$ obtained by flowing by the contact vector field. More precisely we have the Giroux flexibility theorem.

\begin{theorem}(\cite{Gir91})
Let $\Sigma$ be a closed orientable surface, and  $f_0: \Sigma \to (Y, \xi)$ and $g: \Sigma \to (Y', \xi')$ be convex embeddings of $\Sigma$. Suppose $v$ is a contact vector field transverse to $(Y, \xi)$. If the oriented multicurves $f_0^{-1}(\Gamma_{f_0(\Sigma)})$ and $g^{-1}(\Gamma_{g(\Sigma)})$ are isotopic, then there exists an isotopy $f_t$  for $t \in [0,1]$ such that $f_1^*(\xi|_{f_1(\Sigma)}) = g^*(\xi|_{f_1(\Sigma)})$.
\end{theorem}


Given a handlebody, $H$, a \textbf{spine} for $H$ is a graph, $G$, such that $H$ retracts onto $G$. 
If $H$ additionally carries a contact structure, then a spine, $G$, is said to be a \textbf{Legendrian spine} if each edge is a Legendrian arc or knot. By combining Darboux's theorem with the standard neighborhood for Legendrians~\cite[Theorem 2.5.8]{Geiges}, we see that Legendrian spines have a standard tight contact neighborhood, determined by the ribbon neighborhood of the spine tangent to the contact planes. 


\subsection{Contact Heegaard splittings} \label{ss:contactHS}

\begin{definition}
A \emph{contact Heegaard splitting} of a contact manifold $(Y, \xi)$ is a Heegaard splitting $Y=H_1\cup_\Sigma H_2$ such that $H_1$ and $H_2$ are contactomorphic to standard neighborhoods of Legendrian spines $L_1$ and $L_2$.
\end{definition}

We will call the handlebodies which are standard neighborhoods of Legendrian graphs \emph{standard contact handlebodies}. Note that a smooth handlebody of a given genus typically has multiple different ``standard'' contact structures, which are differentiated by the number of components of the dividing set on the boundary. The different options come from the fact that there are generally multiple different surfaces with boundary whose doubles are a genus $g$ surface (determined by the number of boundary components). 

This notion of contact Heegaard splittings originated with Giroux~\cite{Giroux} and was also developed by Torisu~\cite{Tor00}, and can be equivalently formulated as follows.

\begin{lemma}[\cite{Giroux, Tor00}] \label{l:contHeeg}
Let $H_1\cup_\Sigma H_2$ be a Heegaard splitting of $(Y,\xi)$ with $\Sigma$ a convex surface. The following are equivalent
\begin{enumerate}
    \item $H_1$ and $H_2$ are standard neighborhoods of Legendrian graphs
    \item $H_1$ and $H_2$ are two halves of an open book decomposition supporting $(Y,\xi)$
    \item For each handlebody $H_i$, there exists a set of compression disks cutting $H_i$ into a ball, such that the boundary of each compression disk intersects the dividing set on $\Sigma$ in exactly two points.
\end{enumerate}
\end{lemma}

Though these equivalences are known, we review how to pass between them. If we start with $H_1$ and $H_2$ as standard neighborhoods of Legendrian graphs, we can see the page $F$ of the corresponding open book decomposition as a contact framed ribbon of the Legendrian. The contact planes are tangent to $F$ along the Legendrian, so in a standard neighborhood, $d\alpha$ is a positive area form when restricted to this page. Moreover, in the standard model, we can identify $H_i = F\times I/\sim$ (where $(x,t)\sim (x,t')$ for $x\in \partial F$), such that $d\alpha$ is positive on each $F\times \{t\}$. In this way, we see that the first characterization gives rise to the second characterization. Conversely, given a supporting open book decomposition, we can Legendrian realize a spine on two pages. Restricting the contact structure defined on an abstract open book to half of the pages, we see that this is a standard neighborhood of this Legendrian spine. The dividing set on the boundary of a standard neighborhood of a Legendrian graph is precisely the intersection of the contact framed ribbon with $\Sigma=\partial H_i$. The meridian of each edge of the graph thus intersects the dividing set in exactly two points. Similarly, if $H_1$ and $H_2$ are two halves of an open book decomposition, the dividing set on $\Sigma=\partial H_1=\partial H_2$ is the binding of the open book $\partial F$. A collection of compressing disks for $F\times I/\sim$ is given by $a_j\times I/\sim$ where $\{a_j\}$ is a collection of arcs on $F$ which cut $F$ into a disk. Each $a_j\times I/\sim$ intersects the binding in exactly two points (the endpoints of $a_j$ in $\partial F$). If we have a Heegaard splitting with compression disks for each handlebody which each intersect the dividing set in two points, we can cut along these compressing disks to get a ball with the unique tight contact structure. Reversing the cuts amounts to attaching standard contact $1$-handles which gives a standard neighborhood of a Legendrian graph.


There is a fundamental challenge in obtaining compatibility between a multisection and a symplectic structure. Each interior $3$-dimensional handlebody in a multisection appears in the boundary of two $4$-dimensional sectors, but the boundary orientations are opposite to each other. Viewing a sector as a symplectic filling of its boundary induces a contact structure on the boundary which is a \emph{positive} contact structure with respect to the \emph{boundary} orientation. The sign of a contact structure (with respect to a fixed orientation) is an inherent property of the contact planes which measures the direction/handedness of the twisting of the contact planes. Note, this is not the same as the co-orientation of the contact structure which depends only on the contact form, not the contact structure. Therefore, we cannot have identical contact structures on a fixed manifold realize both positive and negative contact structures with respect to a fixed orientation. In general this suggests that we would need two different contact structures on each interior $H_i$ of a multisection. However, as we show in the following lemma, there are both positive and negative \emph{standard contact handlebodies} which are orientation reversing contactomorphic to each other. Both the positive and negative contact structures are supported by the same half open book.


\begin{lemma} \label{l:posnegct}
Let $F$ be a surface with boundary. Let $\eta$ be a $1$-form on $F$ which evaluates positively on the oriented boundary of $F$ such that $d\eta$ is an area form on $F$. Consider $F\times I/\sim$ where $(x,t)\sim(x,t')$ whenever $x\in \partial F$. Using $t$ as the coordinate parametrizing the $I$ direction, let 
\[\alpha^\pm=\pm dt+\eta.\]

Then $\xi^+:=\ker(\alpha^+)$ (respectively $\xi^-:=\ker(\alpha^-)$) is a positive (respectively negative) contact structure on $H$ supported by the trivial half open book on $F\times I/\sim$ (with pages $F\times\{t\}$).
\end{lemma}

\begin{proof}
$\xi^\pm$ is a $\pm$ contact structure because
$$ (\pm dt+\eta) \wedge d(\pm dt+\eta) = \pm dt\wedge d\eta$$
is a $\pm$ volume form on $F\times I$. Note that $\pm dt +\eta$ is independent of $t$ so the form is well-defined on the quotient $F\times I/\sim$.

To check that it is supported by the open book, we need to verify that $\alpha^\pm$ evaluates positively on the binding and $d\alpha^\pm$ is a positive area form on the pages. Indeed if $T$ is a vector positively tangent to $\partial F$, $\alpha^\pm(T)=\eta(T)>0$, and $d\alpha^\pm = d\eta$ is a positive area form on $F$ by assumption.
\end{proof}

\begin{remark}
Note, it is always possible to find such a $1$-form $\eta$ on a surface $F$ with non-empty boundary.
Also, observe the orientation reversing diffeomorphism which sends $t\in I$ to $-t$ takes $\alpha^+$ to $\alpha^-$.
\end{remark}

\subsection{Contact Heegaard diagrams}

Motivated by the third characterization in Lemma~\ref{l:contHeeg}, we define a diagrammatic version of contact Heegaard splittings.

\begin{definition}
A \textbf{contact Heegaard diagram} is a quadruple $(\Sigma, d, C_1, C_2)$ such that:

\begin{itemize}
\item $\Sigma$ is a closed oriented surface.
\item $d$ is a multicurve which separates $\Sigma$ into two homeomorphic surfaces with boundary. 
\item $C_i$ $i \in \{1,2 \}$ is a cut system for $\Sigma$ such that each curve intersects $d$ twice. 
\end{itemize}
\end{definition}

Given a contact Heegaard diagram, we can reconstruct a contact manifold together with a contact Heegaard splitting. In particular, if $F$ is one of the halves of $\Sigma \setminus d$, then we endow each of the handlebodies the standard contact structures coming from $F \times I$. This induces an open book with page $F$ and binding $d$, which in turn induces a contact structure. 
Conversely, every contact Heegaard splitting has a contact Heegaard diagram by characterization (3) of Lemma~\ref{l:contHeeg}.

This correspondence allows us to give a classification of genus-1 contact Heegaard splittings of $S^3$ as we will see in Section \ref{subsec:genus1Classification}. In particular, by Euler characteristic considerations, genus-1 contact Heegaard splittings correspond to open book decompositions of $S^3$ with an annular page. The monodromy then consists or a left-handed or right handed Dehn twist about the core of this annulus. The right handed Dehn twist gives the tight contact structure on $S^3$ whereas the left handed Dehn twist gives an overtwisted structure. These lead to the Heegaard diagrams shown in Figure \ref{fig:genus1S3s}.

\begin{figure}
    \centering
    \includegraphics[scale=.3]{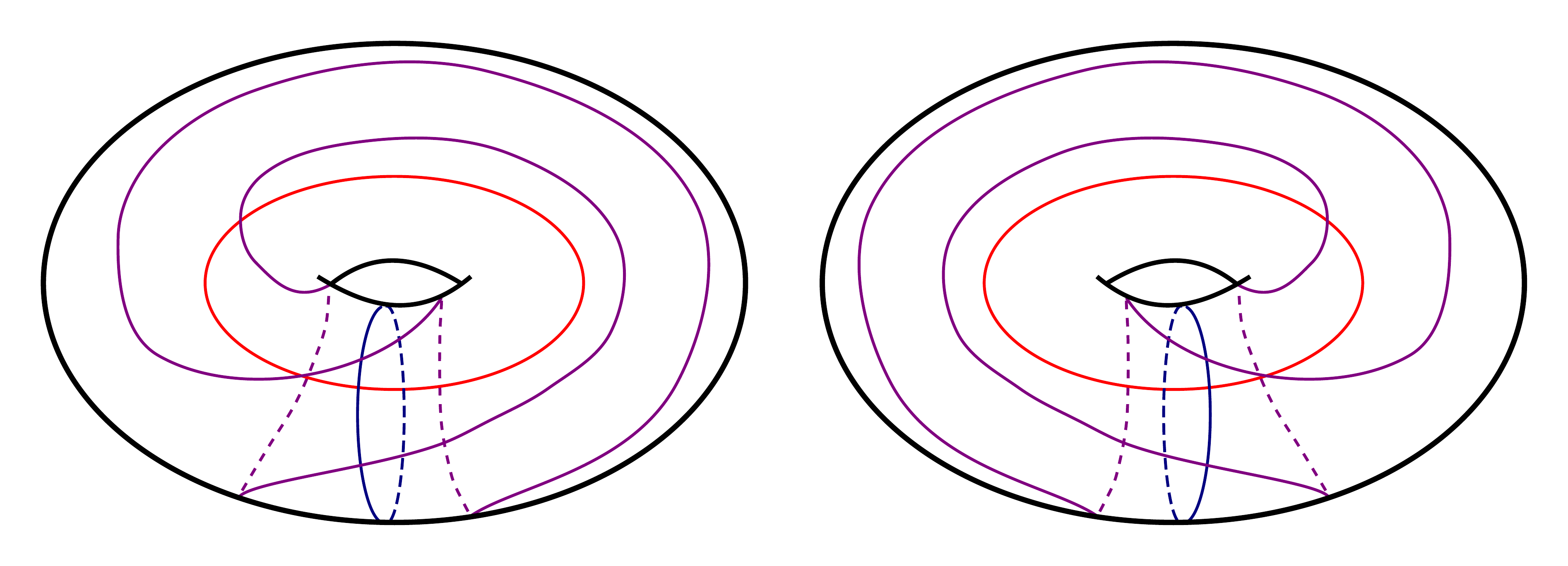}
    \caption{The two genus-1 convex Heegaard diagrams of $S^3$. The one on the left is tight whereas the one on the right is overtwisted.}
    \label{fig:genus1S3s}
\end{figure}

\subsection{Dividing sets on standard neighborhoods from Legendrian front projections}

In Section~\ref{s:KWdiag}, we will be looking at Heegaard splittings of $S^3$, or more generally $\#_{k_1} S^1\times S^2$ where one handlebody is a standard neighborhood of a Legendrian graph $\Lambdatil$ described via a front projection, and the other handlebody is the complement. In order to verify that the complement is a standard contact handlebody, it will be useful to know exactly how to draw the dividing set on the boundary of the standard neighborhood of an explicitly embedded Legendrian graph in terms of the front projection.

We will mainly focus on trivalent Legendrian graphs. Higher valence vertices can be split into trivalent vertices by growing additional Legendrian edges via a Legendrian deformation which preserves the standard neighborhood and thus, the dividing set on its boundary. Given a front projection representing a Legendrian embedding in $\R^3$, we can draw the corresponding dividing set on the boundary of a standard neighborhood. Recall that the dividing set on the boundary of a standard contact handlebody is the boundary of the page of the compatible open book decomposition. A page of the open book is given by the contact framed ribbon of the Legendrian knot. Considering how the contact framing wraps around at left and right cusps and at left and right trivalent vertices, we obtain the local models for the dividing set as shown in Figure \ref{fig:localDividingSetModels}. The first five models cover the generic front projections of a trivalent Legendrian graph. The last model includes a Legendrian arc which degenerately projects to a single point, whose two end points are trivalent vertices. We can isotope this model to a generic front projection as in Figure~\ref{fig:crossings}, and thus derive its local model from the previous models. We include this last ``compound'' model for convenience as we will use it extensively in implementing our algorithm of Section~\ref{ss:existenceKirby}.

\begin{figure}
    \centering
    \includegraphics[scale=.3]{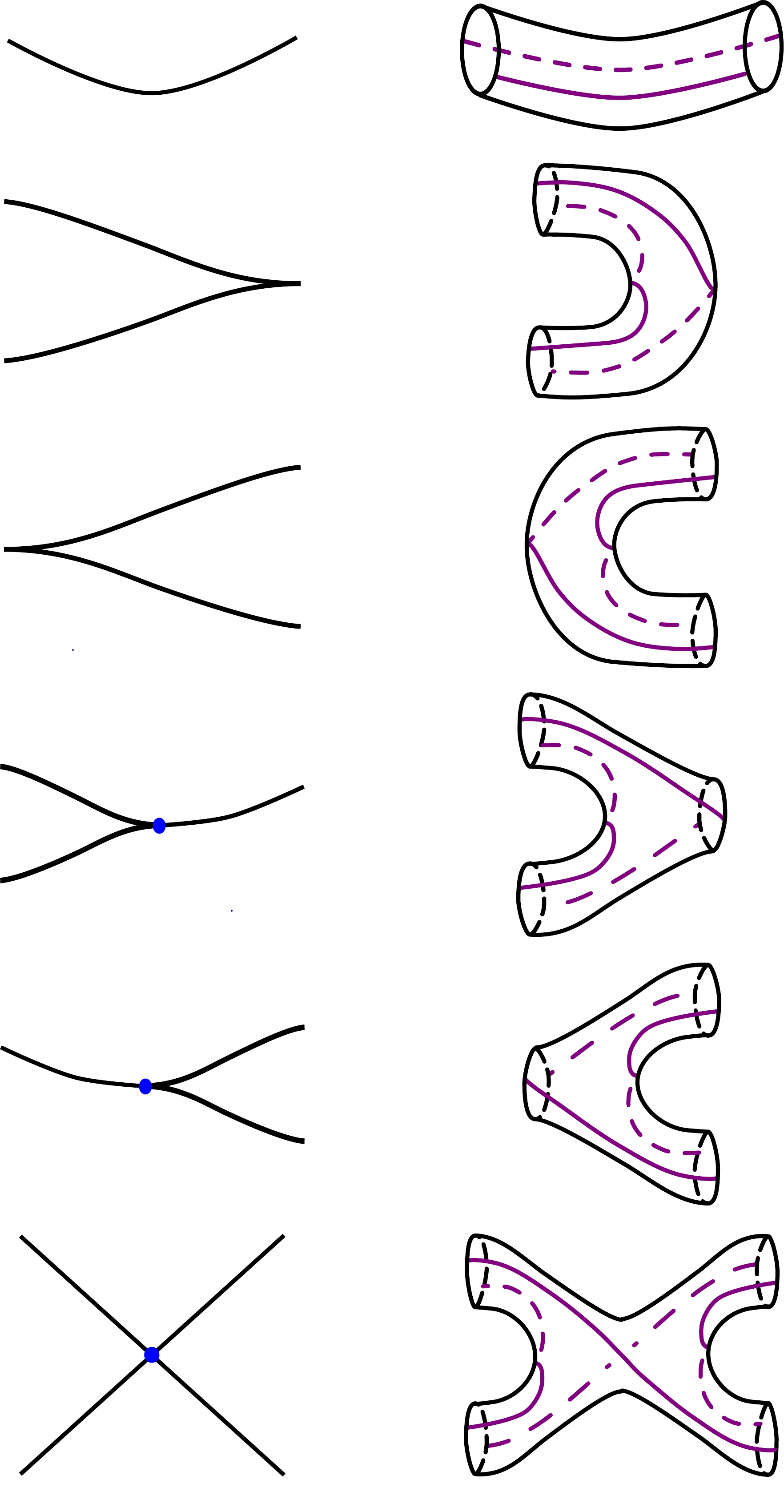}
    \caption{Left: Local models for a trivalent Legendrian graph. Right: Local models for the dividing set of the boundary of a regular neighbourhood of the graph. To obtain the last model from the previous ones, see Figure~\ref{fig:crossings}.}
    \label{fig:localDividingSetModels}
\end{figure}

\section{Kirby-Weinstein handlebody diagrams and multisections with divides} \label{s:KWdiag}

In this section we will show how to use a Kirby-Weinstein handlebody diagram to produce a bisection with divides. A consequence of our proof is an algorithm to obtain a multisection diagram with divides (defined in Section~\ref{ss:WMdiagrams}) from a Kirby-Weinstein diagram.

\subsection{Existence of bisections with divides from Kirby-Weinstein handlebody diagrams}
\label{ss:existenceKirby}

\begin{theorem}
Every compact $4$-dimensional Weinstein domain admits a bisection with divides.
\end{theorem}
\label{thm:existence}

\begin{proof}
By definition, a Weinstein $4$-manifold has a Weinstein handle structure. By~\cite{Gompf98}, this handle structure can be represented in a standard form by a Legendrian front projection with $1$-handles (which we will call a Kirby-Weinstein diagram). We will give an algorithm to convert a Kirby-Weinstein diagram in Gompf standard form into a multisection diagram with divides. If we ignore the symplectic and contact structure, the smooth part of this algorithm essentially follows those in \cite{GayKir16} and \cite{MeiSchZup16} used for converting a Kirby diagram into a trisection.

The union of the $0$- and $1$-handles will be $W_1$. This is diffeomorphic to $\natural_{k_i} S^1\times D^3$ and with the Weinstein structure of $W$ restricted to $W_1$, it is a Weinstein filling of its boundary $(\partial W_1,\xi_1)$.

We will construct a contact Heegaard splitting $H_1\cup_\Sigma H_2$ of $(\partial W_1,\xi_1)$ such that the Legendrian attaching spheres for the $2$-handles of $W$ are a subset of the Legendrian core of $H_2$. Then we will define $W_2$ to be a collar of $H_2$ together with the $2$-handles of $W$. Because the attaching spheres of the $2$-handles are a subset of the Legendrian core of $H_2$, we will see that $W_2$ will be diffeomorphic to $\natural_{k_2} S^1\times D^3$. There is a naturally induced Heegaard splitting of $\partial W_2$ given by $\overline{H}_2 \cup H_3$ where $H_3$ is obtained from $H_2$ by doing Legendrian surgery on the attaching spheres of the $2$-handles. We will then show that this is also a contact Heegaard splitting, and that $(W_2, \omega|_{W_2})$ is a symplectic filling of this contact manifold.

Let $\Lambda$ be the Legendrian attaching link for the $2$-handles of $W$ in $\#_{k_1} S^1\times S^2$. To construct the appropriate contact Heegaard splitting $H_1\cup H_2$ of $(\partial W_1,\xi_1)$, we will add Legendrian tunnels to $\Lambda$, yielding a Legendrian graph $\Lambdatil$ containing $\Lambda$. A standard contact neighborhood of $\Lambdatil$ will be $H_2$ and its complement will be $H_1$. The first purpose of the tunnels is to ensure that the complement of the neighborhood of $\Lambdatil$ is a handlebody. Additional tunnels will be added to ensure this handlebody has the standard contact structure.

The construction of $\Lambdatil$ is as follows.
\begin{enumerate}
    \item Start with $\Lambda$ in Gompf standard form.
    \item If there is any $1$-handle of $W$ whose belt sphere is disjoint from $\Lambda$, add a Legendrian circle which passes through that $1$-handle once.
    \item For each $1$-handle add Legendrian arcs to connect all the strands that pass through that $1$-handle on the left and right as shown in Figure~\ref{fig:leggraph1h}.
    \item At each crossing in the diagram, add a Legendrian arc which projects to the crossing point connecting the over- and under-strands. See Figure~\ref{fig:crossings} for generic front projections for a Legendrian isotopic graph.
    \item Add Legendrian arcs to connect disconnected components until the graph is connected.
    \item The resulting front projection divides up the plane into regions, such that the boundary of each region is a Legendrian unknot. Add further Legendrian arcs to cut up each region as in Figure~\ref{fig:algorithmTunnels} so that at the end, each region is bounded by a Legendrian unknot with $tb=-1$. Namely, every region in the front projection should have a unique ``right cusp'' and a unique ``left cusp'' where a vertex is a right (resp. left) cusp of a region if the two edges on the boundary of the region which meet in the vertex both approach the vertex from the left (resp. right). Note that here we can treat each crossing as a single vertex by shrinking the Legendrian arc connecting the two strands by a Legendrian deformation.
\end{enumerate}

\begin{figure}
    \centering
    \includegraphics[scale=.5]{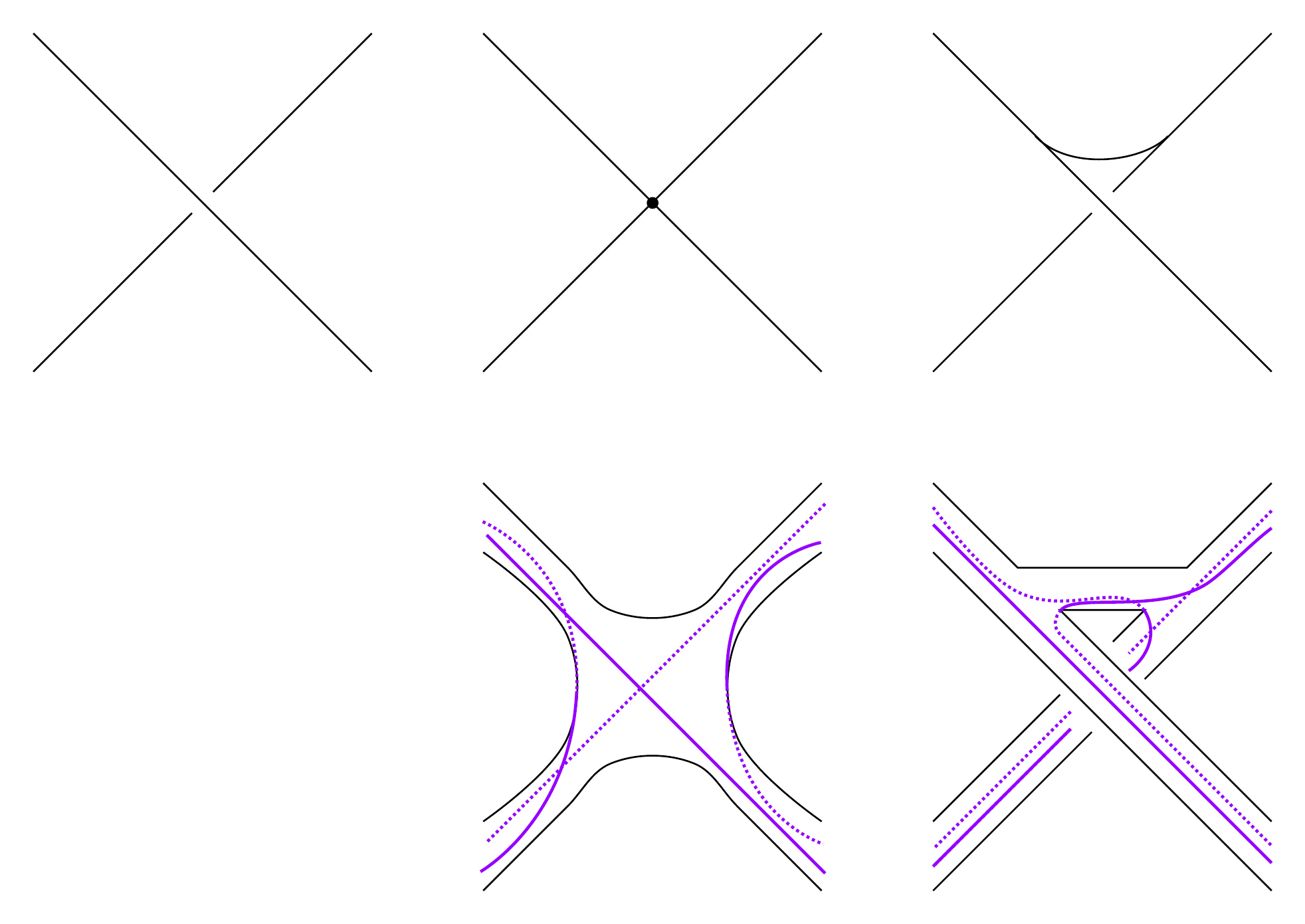}
    \caption{Adding a tunnel at each crossing which projects to the crossing point (top center) is Legendrian isotopic to the generic front projection shown on the top right. The dividing set on the boundary of a neighborhood can be determined using the standard models to obtain the lower right picture, and then isotoped to obtain the lower center picture.}
    \label{fig:crossings}
\end{figure}

\begin{figure}
    \centering
    \includegraphics[scale=.3]{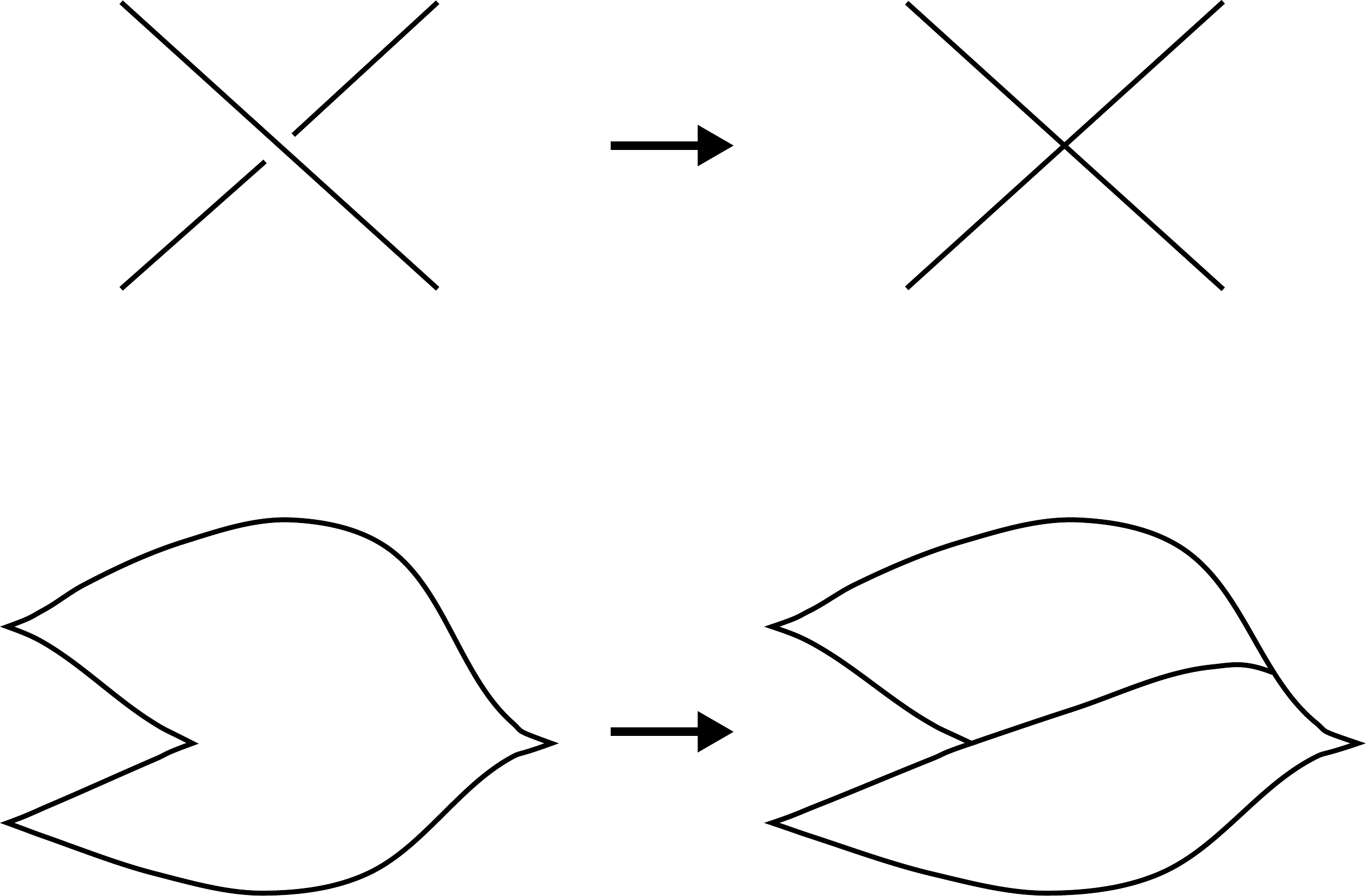}
    \caption{We first flatten each crossing in the knot by asdding a vertical Legendrian tunnel. To make each region a Legendrian unknot with $tb=-1$ we can further add tunnels to regions with additional cusps.}
    \label{fig:algorithmTunnels}
\end{figure}

\begin{figure}
    \centering
    \includegraphics[scale=.8]{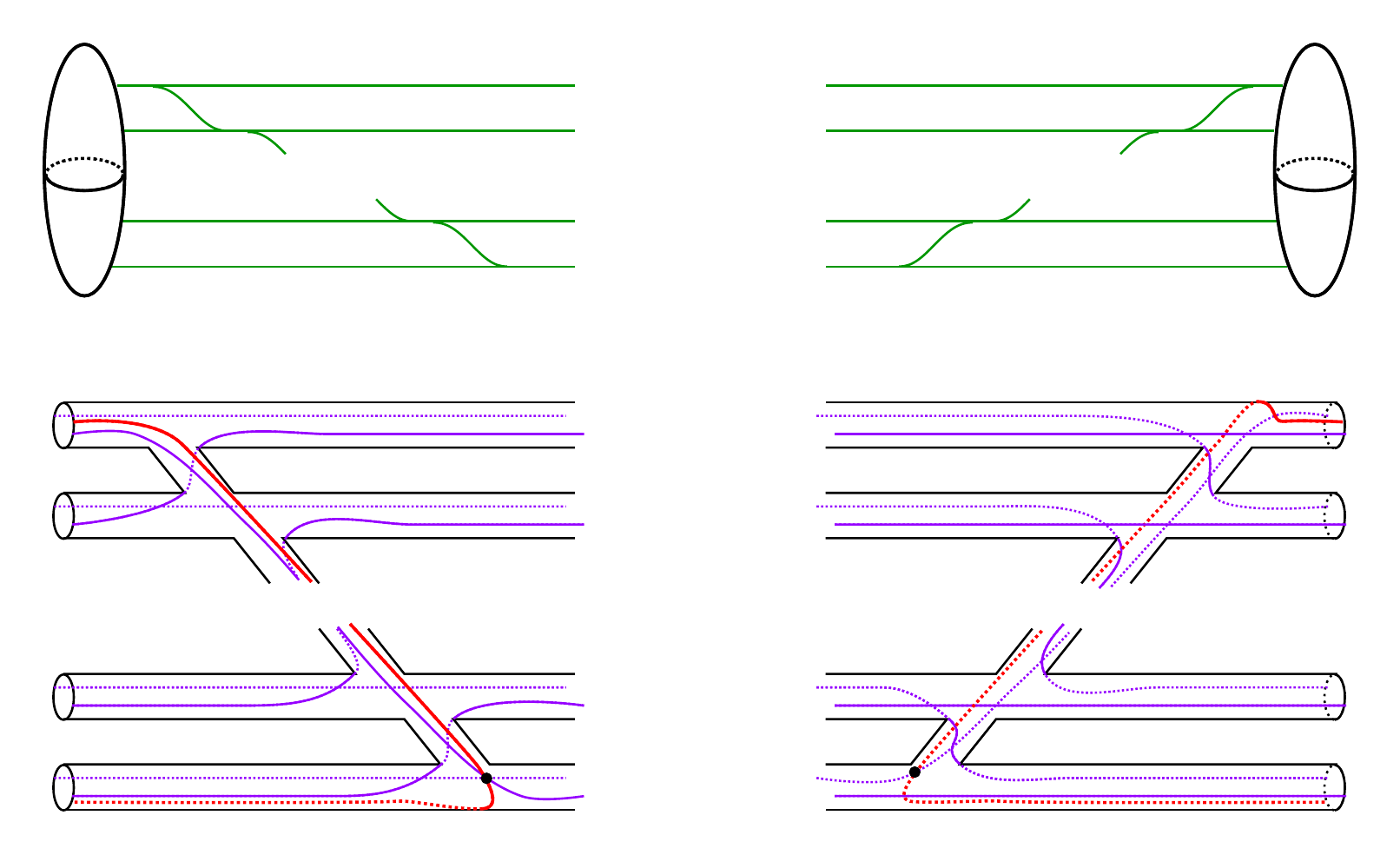}
    \caption{Adding tunnels between parallel strands passing through the same $1$-handle allows us to locate a curve on the boundary of the neighborhood of the Legendrian graph $\Lambdatil$ bounding a disk on the exterior coming from the surgery induced by the $1$-handle. Choosing the arcs in precisely the manner shown ensures that we can find such a curve intersecting the dividing set in exactly two points. The lower part of the figure shows the relevant portion of the Heegaard surface with the dividing set (in purple) and the compressing curve (in red).}
    \label{fig:leggraph1h}
\end{figure}


Now if $H_2$ is a standard contact neighborhood of $\Lambdatil$, and $H_1$ is the complement, we want to show that $H_1\cup H_2$ is a contact Heegaard splitting. If there are no $1$-handles in the Kirby-Weinstein diagram, $H_1$ is a smooth handlebody because we put tunnels at each crossing, so $H_1$ is diffeomorphic to a $3$-ball with one handle attached for each bounded planar region in the diagram of $\Lambdatil$. When there are $1$-handles in the Kirby-Weinstein diagram, there is an additional compressing disk for each $1$-handle as seen in Figure~\ref{fig:1handle}. 

\begin{figure}
    \centering
    \includegraphics[scale=.7]{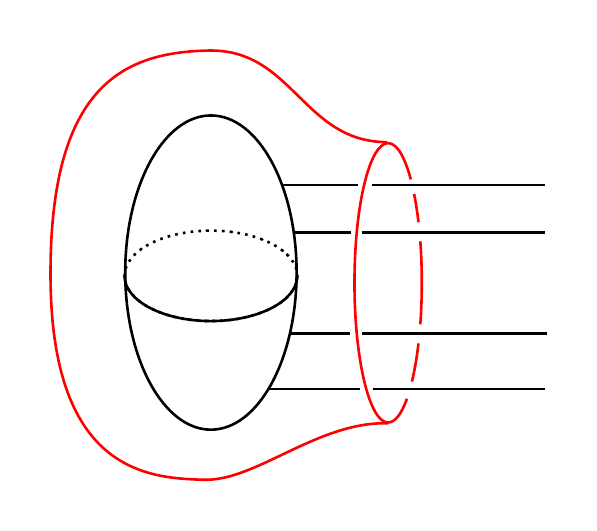}
    \caption{Additional compressing disk for each $1$-handle.}
    \label{fig:1handle}
\end{figure}

Using the models from Figure~\ref{fig:localDividingSetModels}, we can draw the dividing set $d$ on the boundary $\Sigma$ of $H_2$ in terms of the front projection. 
To see that $H_1$ has the standard contact structure, we show that there is a set of compressing curves for $H_1$ on $\Sigma$ such that each curve intersects the dividing set in two points.

Let $\{D_j'\}$ be the bounded regions of the complement of the front projection of $\Lambdatil$, and let $D_j=D_j'\cap H_1$. If there are no $1$-handles in the handle diagram for $W$, then $\{D_j\}$ form a collection of compressing disks for $H_1$, which cut $H_1$ into a ball. If there are $1$-handles, there is also a compressing disk as in Figure~\ref{fig:1handle} for each $1$-handle. After performing handleslides over the regions that pass through the $1$-handle, we can realize this compressing curve as in Figure~\ref{fig:leggraph1h} so that it intersects the dividing set in exactly two points. To see that that boundary of each $D_j$ intersects the dividing set in two points, we use the property from the last step of the construction of $\Lambdatil$, that each region is bounded by a Legendrian unknot with $tb=-1$, meaning there is a unique left cusp (which may be a vertex) and a unique right cusp (which may be a vertex). Any other vertices along the boundary of the region have the two edges on the boundary of this region entering the vertex from different sides. Examining all of the ways that our local models in Figure~\ref{fig:localDividingSetModels} may appear as the boundary of a region, we see that each cusp (either a standard cusp or a vertex cusp) contributes one intersection point between the dividing set $\partial D_j$, and the remaining edges and vertices in the boundary of the region do not contribute any intersections between the dividing set and $\partial D_j$. Thus we see that $H_1$ is a standard contact handlebody.

Next, we look at the second sector $W_2$. We need to show that $W_2\cong \natural_{k_2} S^1\times D^3$, that there is a contact structure $\xi$ on $\partial W_2$ with a contact Heegaard splitting $\overline{H_2}\cup_\Sigma H_3$ such that $(W_2,\omega|_{W_2})$ is a symplectic filling of $(\partial W_2,\xi)$. Recall that $W_2$ is obtained from $H_2\times I$ by attaching $2$-handles along the Legendrian link $\Lambda\subset \Lambdatil$ with framing $ct-1$ where $ct$ is the framing induced by the contact planes. Since $\Lambda$ is embedded in $\Lambdatil$ (the core of $H_2$) $W_2$ is smoothly diffeomorphic to $\natural_{k_2}S^1\times D^3$. A natural Heegaard splitting of $\partial W_2$ is given by $\overline{H_2}\cup_\Sigma H_3$ where $H_3=(H_2)_{ct-1}(\Lambda)$ ($ct-1$ surgery of $H_2$ along $\Lambda$). There is a well-defined contact structure obtained by $ct\pm 1$ Legendrian surgery, which agrees with the contact structure on $H_2$ near $\partial H_3=\Sigma$ (since the surgery is performed on the interior). Therefore the dividing set on $\partial H_3=\Sigma$ can be viewed as the same as the dividing set on $\partial H_2=\Sigma$, but the compressing curves for $H_3$ change based on the surgery. Namely, for each component $\Lambda_k$ of $\Lambda$, the compressing curve changes from a meridian of $\Lambda_k$ to a $ct-1$-framed copy of $\Lambda_k$ on $\Sigma$. Since the dividing set is parallel to the $ct$ framing of $\Lambda_k$, the $ct-1$ framing of $\Lambda_k$ intersects the dividing set exactly twice. Therefore, $H_3$ is a standard contact handlebody with the contact structure induced by Legendrian surgery on $H_2$.. Note that $\overline{H_2}$ is also a standard contact handlebody, using the negative contact structure on $H_2$ from Lemma~\ref{l:posnegct}. Putting these together, we get a contact Heegaard splitting of $\partial W_2$.

It suffices to show that $(W_2,\omega|_{W_2})$ is a Weinstein filling of $\partial W_2$ where the contact structure on $\partial W_2$ is given by the contact Heegaard splitting $\overline{H_2}\cup_\Sigma H_3$. For this, notice that $H_2\times I$ is a $1$-handlebody, and if we restrict $\omega$ to this subset of $W_2$, up to shrinking $I$, the symplectic structure must be a standard neighborhood of the isotropic spine of $H_2$. In other words, $H_2\times I$ with the symplectic structure $\omega$ is symplectomorphic to a Weinstein $1$-handlebody. Moreover, the induced unique tight contact structure on $\partial (H_2\times I)=\#_{k_1}S^1\times S^2$ is supported by the open book $F\times S^1$ (viewing $H_2=F\times I$). This open book with trivial monodromy gives rise exactly to the contact Heegaard splitting $\overline{H_2}\cup H_2$ (where we collapse $I$ along $\partial H_2$). Since $W_2$, (with the contact Heegaard splitting $\overline{H_2}\cup_\Sigma H_3$) is obtained from the Weinstein $1$-handlebody $H_2\times I$ (with the contact Heegaard splitting $\overline{H_2}\cup_\Sigma H_2$) by attaching Weinstein $2$-handles along Legendrian knots in $H_2$, we have that $W_2$ is a Weinstein filling of the contact Heegaard splitting $\overline{H_2}\cup_\Sigma H_3$.

\end{proof}

\subsection{Multisection diagrams with divides} \label{ss:WMdiagrams}

A fundamental feature of multisections with divides is that they can be completely encoded on a surface. In this section we define these diagrams and show how the previous proof gives an algorithm for obtaining a bisection diagram with divides from a Kirby-Weinstein handlebody diagram.

\begin{definition}
A \textbf{multisection diagram with divides} is a closed orientable surface, together with a set of dividing curves $d$ and cut systems $C_1, C_2, ... , C_n$ such that for all $i \in \{1, 2, ... , n-1 \}$
\begin{itemize}

    \item Each curve in $C_i$ intersects $d$ in two points.
    \item $(\Sigma, d, C_i, C_{i+1})$ is a contact Heegaard splitting of the \emph{tight} contact structure on $\#_{k_i} S^1 \times S^2$ for some $k_i \in \mathbb{N}$. 
\end{itemize}
\end{definition}

\begin{remark}
    The condition that each curve in each cut system intersects the dividing set in two points is easily checked. However, it is potentially difficult to check whether the union of two consecutive contact handlebodies forms the \emph{tight} contact structure on $\#_{k_i} S^1\times S^2$.
\end{remark}

The proof of Theorem~\ref{thm:existence} gives an algorithmic method to obtain a multisection diagram with divides as follows.

Starting from a Kirby-Weinstein handlebody diagram, construct the Legendrian graph $\Lambdatil$ as in the proof. Use the models from Figure~\ref{fig:localDividingSetModels} to draw the dividing set $d$ on $\Sigma$ as the boundary of the neighborhood of $\Lambdatil$ ($H_2$) in terms of the front projection.

We can describe cut systems $C_i$ $i\in \{1,2,3\}$ for the handlebodies $H_1$, $H_2$, and $H_3$ as follows. The cut system $C_1$ is given by taking the regions of the planar diagram together with a curve tracing through the vertical tunnels added between the 2-handles passing through each 1-handle, as in Figure~\ref{fig:leggraph1h}--note the curve in this figure intersects the dividing set in two points. As the regions are diagrams of $tb=-1$ unknots, each of the curves given by the boundary of a region intersects the dividing set twice. This ensures that the pair $(\Sigma, C_1)$ is a diagram of a standard contact handlebody.

The cut system $C_2$ is given by taking a meridian of each tunnel together with a meridian of each knot in the Kirby-Weinstein diagram. Using the local model at the top of Figure \ref{fig:localDividingSetModels} and Figure~\ref{fig:crossingtun} we see that these curves intersect the dividing set twice so that $(\Sigma, C_2)$ is also a standard contact handlebody. Because $(\Sigma,C_1,C_2,d)$ represents a Heegaard splitting of the boundary of the $0$- and $1$-handles of the Kirby-Weinstein handlebody diagram, it is symplectically fillable and thus must support the tight contact structure on a connected sum of copies of $S^1 \times S^2$.

\begin{figure}
    \centering
    \includegraphics[scale=.8]{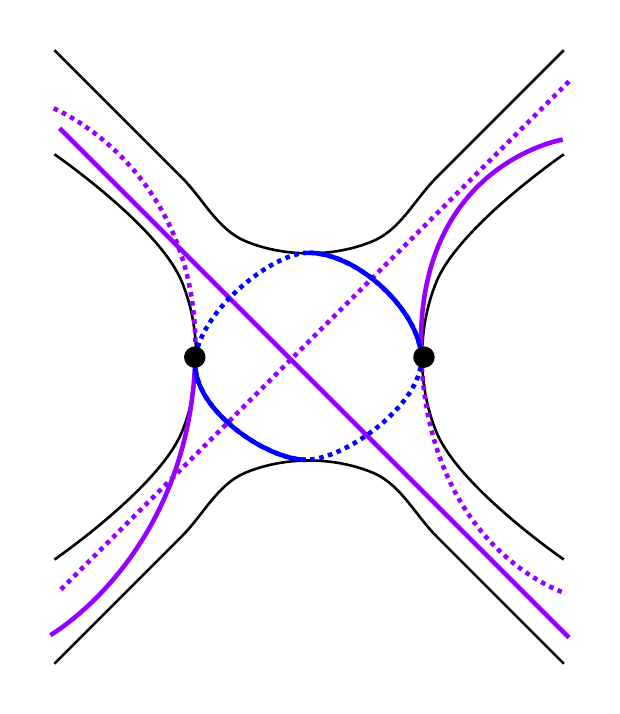}
    \caption{The meridian of the tunnel at a crossing projects as the blue curve in the figure. Note this curve intersects the dividing set at two points.}
    \label{fig:crossingtun}
\end{figure}

To obtain $C_3$, we start by taking a Legendrian push off of each knot component $K_i$ in the diagram. Each such component intersects its chosen meridian in $C_2$ once and does not intersect the dividing set at all. Adding a left handed twist about the meridian to each component $K_i$ gives a curve $K_i'$ whose surface framing is the contact framing $-1$ and which intersects the dividing set twice. Then replacing each meridian in $C_2$ with the corresponding $K_i'$ gives a diagram $(\Sigma, d, C_2, C_3)$ representing a contact Heegaard splitting of a connected sum of copies of $S^1 \times S^2$. Therefore $(\Sigma,d,C_1,C_2,C_3)$ is a bisection diagram with divides of the given manifold.

This algorithm is carried out in Figures \ref{fig:legTrefoilAndTunnels} and \ref{fig:bisectionOfTrefoilTrace} for the result of attaching a Weinstein handle to the max $tb$ right handed trefoil and in Figures \ref{fig:TstarRP2Tunnels} and \ref{fig:TstarRP2Bisection} for the Weinstein domain $T^* \rptwo$, which is a disk bundle over $\rptwo$. 

\begin{figure}
    \centering
    \includegraphics[scale=.23]{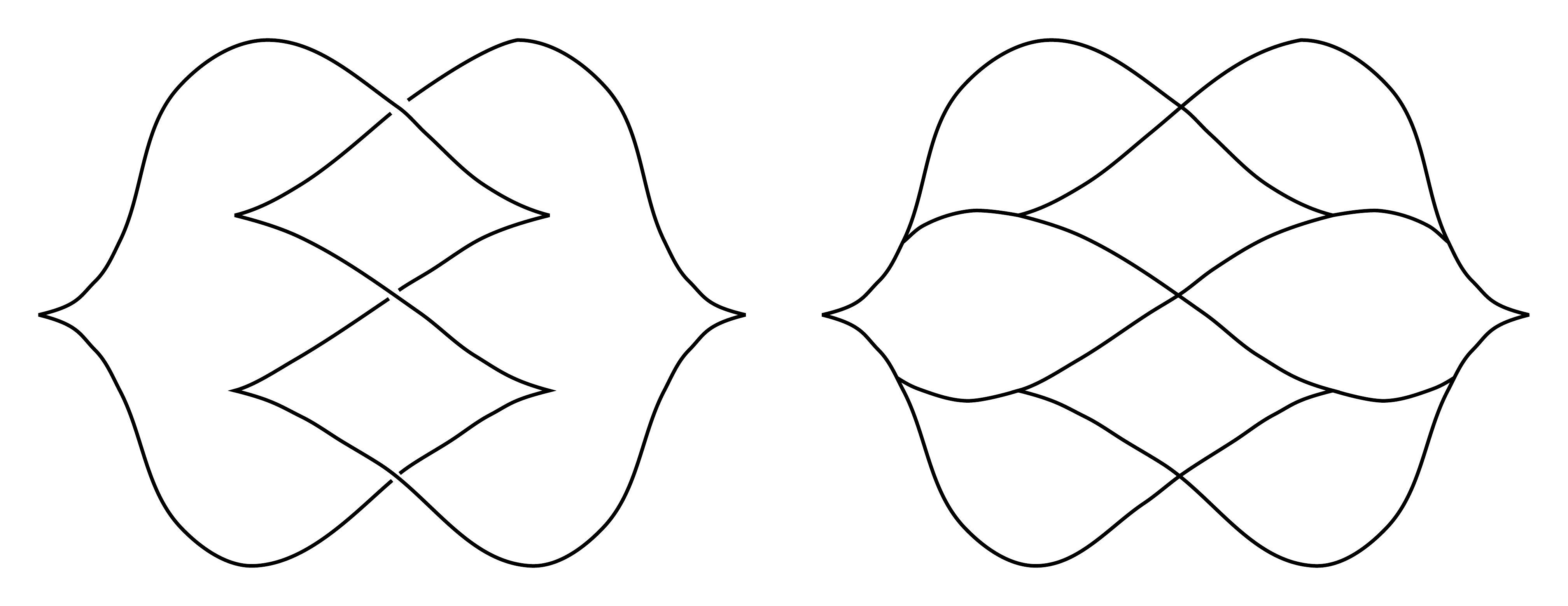}
    \caption{Adding tunnels to the right handed Legendrian trefoil on the left results in the graph on the right. Note that the exterior is unknotted and that each bounded region in the diagram is a $tb=-1$ unknot. }
    \label{fig:legTrefoilAndTunnels}
\end{figure}

\begin{figure}
    \centering
    \includegraphics[scale=.23]{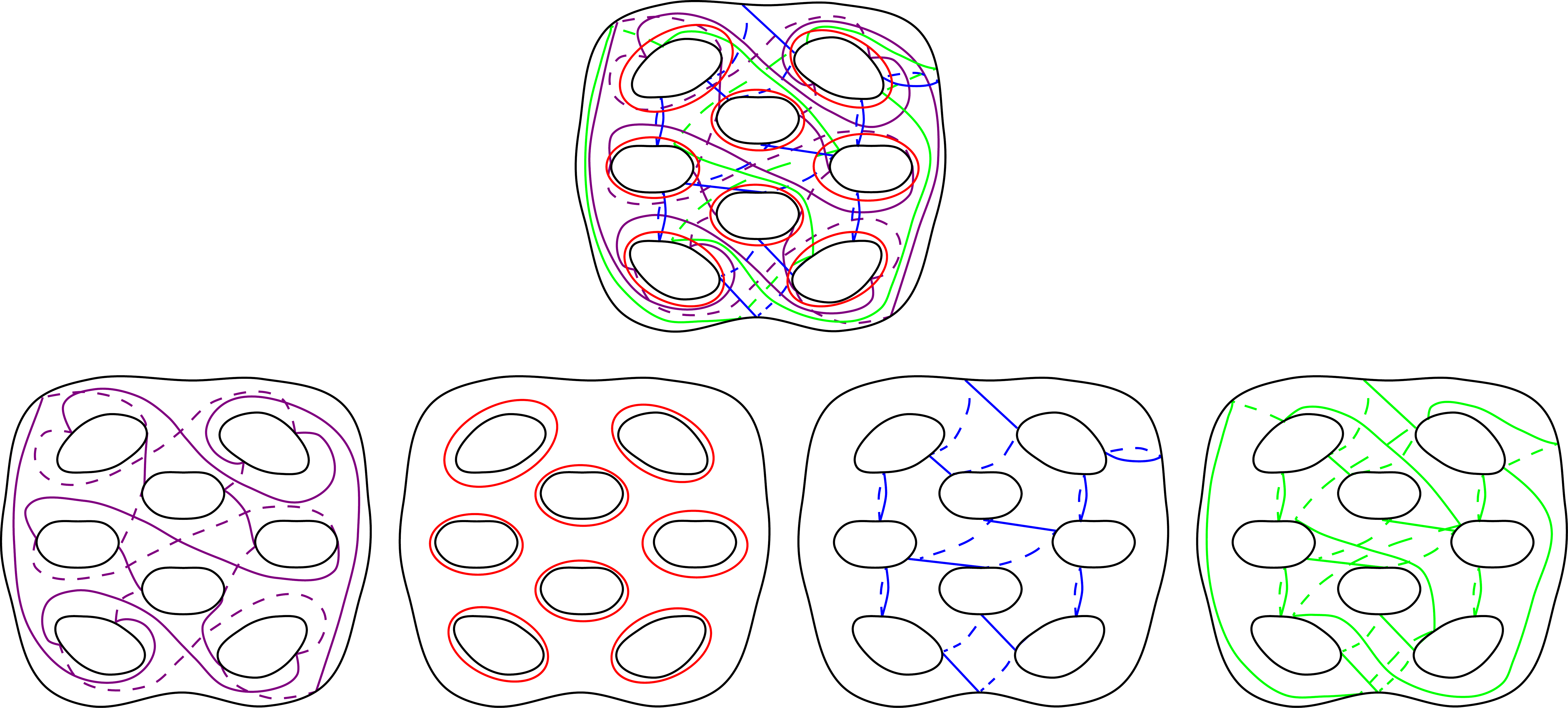}
    \caption{Top: A bisection diagram of the manifold obtained by attaching a Weinstein handle to the max $tb$ trefoil. On this diagram, many green curves which are parallel to blue curves are omitted for visual clarity. Bottom: The dividing set, followed by each of the cut systems in order are drawn out individually.}
    \label{fig:bisectionOfTrefoilTrace}
\end{figure}

\begin{figure}
    \centering
    \includegraphics[scale=.23]{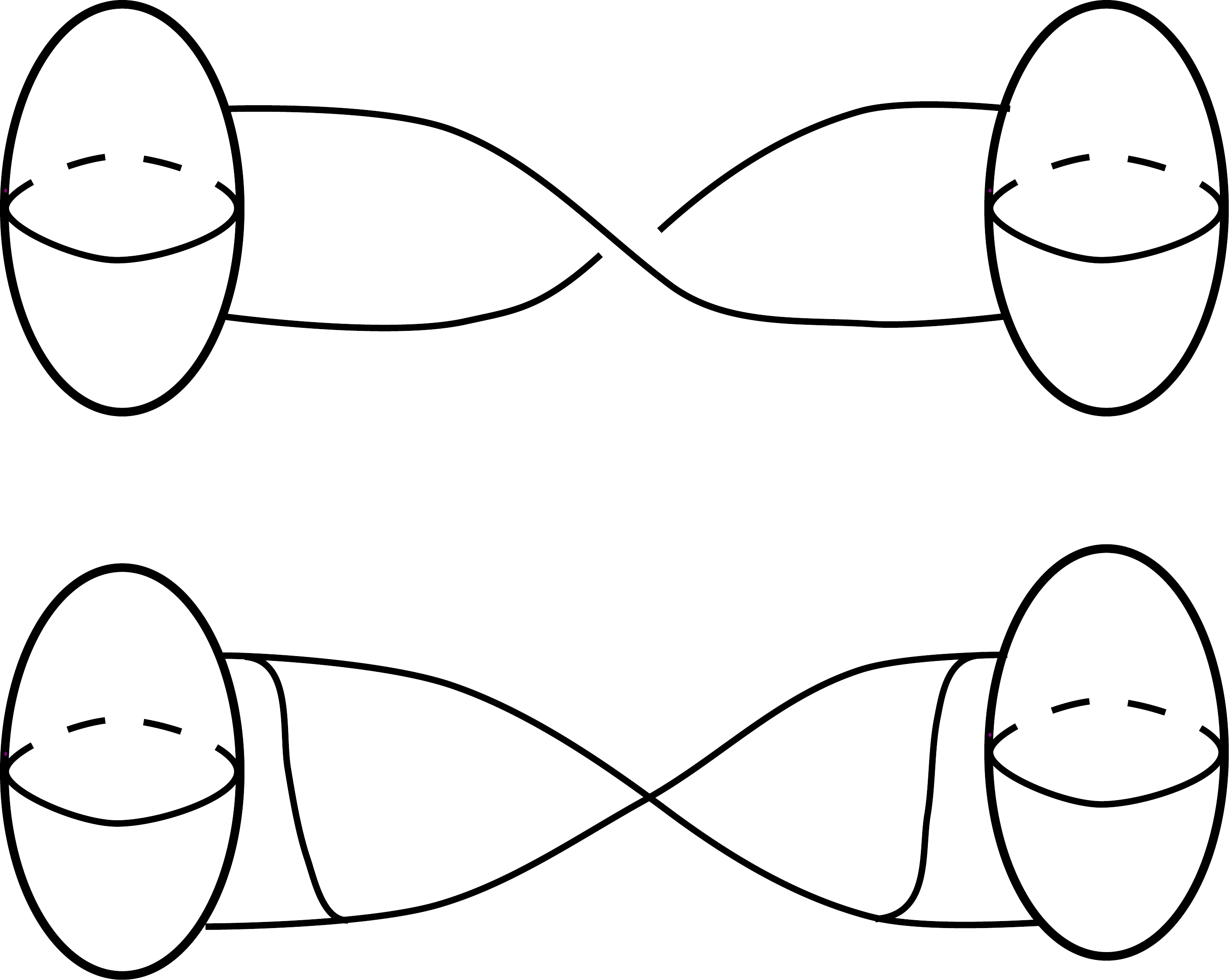}
    \caption{Top: A Kirby-Weinstein diagram for $T^* \mathbb{R}P^2$. Bottom: Following the proof of Theorem \ref{thm:existence} we add Legendrian tunnels to the diagram above so that the exterior of the tunnels drawn is a standard contact handlebody. }
    \label{fig:TstarRP2Tunnels}
\end{figure}

\begin{figure}
    \centering
    \includegraphics[scale=.15]{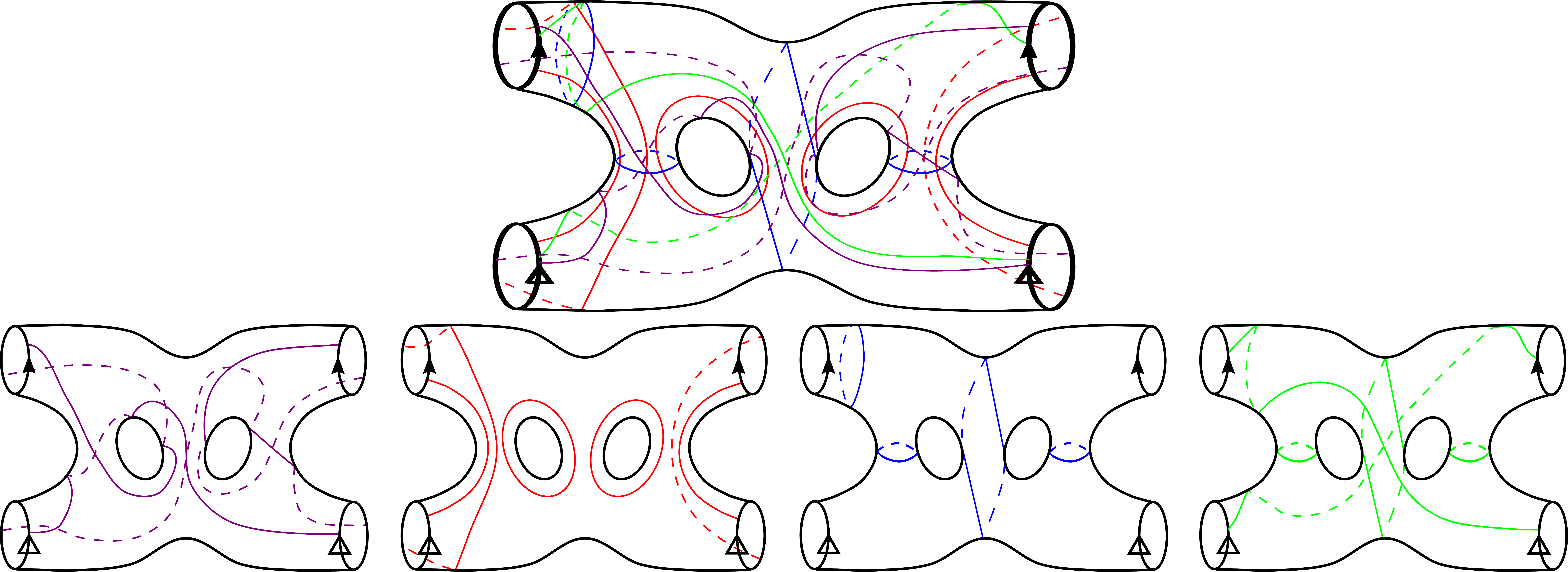}
    \caption{Top: A bisection diagram of $T^* \mathbb{R}P^2$. On this diagram, many green curves which are parallel to blue curves are omitted for visual clarity. Bottom: The dividing set, followed by each of the cut systems in order are drawn out individually.}
    \label{fig:TstarRP2Bisection}
\end{figure}

Note that the cut systems $C_2$ and $C_3$ which are output from our algorithm have a very particular form. More specifically, each component $\gamma$ of $C_3$ either agrees with or is dual to a component $\beta$ of $C_2$. In the latter case, there exists a curve $V$ in $\Sigma$ which is disjoint from the dividing set, dual to the component $\beta$ with respect to $C_2$, such that $\gamma=\tau^{-1}_\beta(V)=\tau_V(\beta)$ where $\tau_A$ is a right-handed Dehn twist about $A$ with respect to the orientation on $\Sigma$ induced as the boundary of $H_2$. Let's call two cut systems related in this way \emph{standard Weinstein cobordant}. By isotoping each $V$ into the Legendrian core of the handlebody $H_2$, we find a Legendrian link in $H_2$ such that the $ct-1$ framing for this link is represented by the corresponding $\gamma$ curves in $C_3$, ($V$ represents the contact framing, and $\beta$ is a meridian of the surgery torus so $\tau^{-1}_\beta(V)$ represents the contact framing $-1$). Thus we see that if two cut systems with a dividing set are standard Weinstein cobordant, the corresponding sector can be endowed with the structure of a Weinstein cobordism from $H_i$ to $H_{i+1}$ obtained by attaching Weinstein $2$-handles to $H_i\times I$. We can always endow the first sector with the structure of a Weinstein $1$-handlebody (since by definition $(\Sigma,d,C_1,C_2)$ is a contact Heegaard splitting of $\#_{k_1} S^1\times S^2$ with the tight (fillable) contact structure). Therefore, we have an (overly strong) condition that ensures a multisection diagram with divides corresponds to a Weinstein $4$-manifold.

\begin{prop}
\label{prop:standardDiagramsAreWeinstein}
Let $(\Sigma,C_1,\dots, C_n,d)$ be a multisection diagram with divides such that $(C_i,C_{i+1},d)$ are standard Weinstein cobordant for $1<i<n$. Then $(\Sigma,C_1,\dots, C_n,d)$ corresponds to a Weinstein $4$-manifold.
\end{prop}

In Section~\ref{s:PALF}, we will see another condition that ensures a multisection diagram with divides corresponds to a Weinstein $4$-manifold. It is an interesting question to ask whether there is a general characterization of all multisection diagrams with divides which correspond to Weinstein $4$-manifolds. In general, we only expect a multisection diagram with divides to correspond to a symplectic $4$-manifold, which may not admit any global Weinstein structure.

\section{PALFs, monodromy substitution and multisections with divides} \label{s:PALF}
\subsection{PALFS and monodromy factorizations}
Fibration structures on symplectic manifolds have a long history of study, dating to Donaldson's work in~\cite{Don99} where it was shown that every closed symplectic 4-manifold admits a Lefschetz pencil. Conversely, Gompf proved that every 4-manifold with a Lefschetz pencil admits a symplectic structure~\cite{GompfStipsicz}. The corresponding objects in the Weinstein category are positive allowable Lefschetz fibrations.

\begin{definition}
A \emph{Lefschetz fibration} on $X$ is a map $\pi: X\to B$ to a surface $B$ such that near each critical point of $\pi$, there are local orientation preserving coordinates such that $\pi$ is modeled on $(z_1,z_2)\mapsto z_1z_2$. A \textbf{positive allowable Lefschetz fibration} (PALF) is a Lefschetz fibration whose base $B$ is $D^2$ and whose regular fiber is a compact surface with boundary such that every vanishing cycle is homologically essential (allowable). 
\end{definition}

Following Loi and Piergallini \cite{LoiPie01}, every Weinstein domain admits a PALF. Conversely, every PALF supports a Weinstein structure~\cite{GompfStipsicz}. 

In this section, we will show how to use the PALF structure to obtain a decomposition of a Weinstein domain as a multisection with divides. As in Figure~\ref{fig:palfToMultisection}, we cut the disk $D^2$ into closed subdisks, $D_1,\dots, D_n$, such that 
\begin{itemize}
    \item Each $D_i$ contains a unique critical value of $\pi$,
    \item $D_i\cap D_{i+1}$ is diffeomorphic to an interval for $i=1,\cdots, n-1$,
    \item $\partial D^2\subset X_1\cup X_n$, and
    \item $D_1\cap \cdots \cap D_n = \{(0,1),(0,-1)\}$.
\end{itemize}
Then let $X_i = \pi^{-1}(D_i)$, $H_1 = \pi^{-1}(\partial D^2\cap D_1)$, $H_i = \pi^{-1}(D_{i-1}\cap D_i)$ for $i=2,\dots, n$, $H_{n+1}=\pi^{-1}(\partial D^2\cap D_n)$. Then $H_i = X_{i-1}\cap X_i$ and $X_1\cap \cdots \cap X_n = \partial H_i$ (after quotienting by the $D^2$ factor along points in $\partial F\times D^2$ which is a Weinstein homotopic domain). 

\begin{theorem} \label{thm:palfToMultisection}
 The decomposition $(X_1,\dots, X_n)$ is a multisection with divides for $X$.
\end{theorem}

\begin{proof}
Each $H_i$ is a $3$-dimensional handlebody since it is diffeomorphic to $F\times I$. Thus to check this is a multisection with divides it suffices to check that (1) each $X_i$ is symplectomorphic to a $4$-dimensional Weinstein $1$-handlebody and (2) that $H_i\cup H_{i+1}$ is a contact Heegaard splitting of $\partial X_i$. 

First we look at each $X_i$. We will use $F:=\pi^{-1}((0,1))$ as the regular fiber. The vanishing cycles are curves $(c_1,\dots, c_n)$ in $F$ which collapse to the critical point under parallel transport from $(0,1)$ to the critical value.
The model for Lefschetz singularities shows that $X_i$ is diffeomorphic to the manifold obtained from $H_i\times I$ by attaching a $2$-handle along $c_i$ with framing given by one less than the page framing. $H_i\times I$ is certainly a $1$-handlebody, so the result will still be a $4$-dimensional $1$-handlebody if this $2$-handle cancels with one of the $1$-handles of $H_i\times I$. Thus to see that $X_i$ is diffeomorphic to a $1$-handlebody it suffices to check that there is a meridional disk in $H_i$ which intersects $c_i$ exactly once.

As $\pi$ is a PALF, $c_i \subset F$ is homologically essential. Thus, by Poincare-Lefschetz duality, there exists an arc $a_i\subset F$ such that $|a_i \cap c_i| = 1$. Parallel transport defines diffeomorphisms $\Psi_i: F\times I \to H_i$ which identify $F\times \{0\}$ with $F\subset H_i$. Then $\Psi_i(a\times I)$ is a meridional disk of $H_i$ which intersects $c_i\subset F=\Psi_i(F\times \{0\})$ at a single point. Thus $X_i$ is a $4$-dimensional $1$-handlebody.

Since $X_i=\pi^{-1}(D_i)$ and $D_i$ is a disk, $X_i$ has the structure of a Weinstein manifold induced by the (restricted) PALF. This agrees with the symplectic structure on $X$ since both are compatible with the Lefschetz fibration.

To see that $H_i\cup H_{i+1}$ gives a contact Heegaard splitting of $\partial X_i$ with the contact structure induced by the Weinstein structure on $X_i$, we use the open book construction of contact Heegaard splittings. Restricting $\pi$ to $\partial X_i$ gives an open book decomposition of $\partial X_i$ which supports the contact structure induced by the Weinstein structure on $X_i$ since the Weinstein structure comes from the Lefschetz fibration structure. $H_i$ and $H_{i+1}$ are precisely the two halves of the open book which give a contact Heegaard splitting. 
\end{proof}

\begin{remark}
Using the fact that all Weinstein manifolds are supported by Lefschetz fibrations, Theorem~\ref{thm:palfToMultisection} gives a similar result to Theorem~\ref{thm:existence}. The main difference is that Theorem~\ref{thm:existence} yields a \emph{bisection}, whereas Theorem~\ref{thm:palfToMultisection} will usually have many more than two sectors.
\end{remark}

\subsection{Multisection diagrams with divides from a PALF}
A PALF can be encoded combinatorially through the fiber surface $F$ and the ordered set of vanishing cycles $(c_1,\dots, c_n)$. In this section we show how to use the combinatorial data of a PALF to obtain the combinatorial data of the multisection diagram with divides corresponding to the decomposition from Theorem~\ref{thm:palfToMultisection}. 

The monodromy about a Lefschetz critical value with vanishing cycle $c$ is a right handed Dehn twist about $c$, which we denote by $\tau(c)$. Thus a PALF can equivalently be encoded by an ordered sequence of right handed Dehn twists about the vanishing cycles called a \textbf{monodromy factorization}. 

The core surface of the multisection $\Sigma$ is diffeomorphic to the union of two copies of $F$ glued together along their boundary. The dividing set on $\Sigma$ is given by the boundary of $F$. More precisely, if $\Psi_i: F\times I\to H_i$ is the diffeomorphism defined by parallel transport, $\Sigma = \Psi_i (F\times \{0\} \cup F\times \{1\}) = F\cup \Psi_i(F\times \{1\})$ (note we are suppressing the quotient of the $I$ direction at points in $\partial F$).

To understand a multisection diagram with divides, we want to fix an identification of $\Sigma$, and then draw cut systems for each handlebody $H_i$. We will use $\Psi_1$ to identify $\Sigma=\partial H_1$ as $F_0\cup F_1$. Then the restriction of $\Psi_i$ to $F_0\cup F_1$ gives a diffeomorphism from $\Sigma$ to $\partial H_i$.

Let $\{a_1,\dots, a_k\}$ be a complete arc system for $F$ i.e. a collection of properly embedded arcs which cut $F$ into a disk.  Then $\{\Psi_i(a_1\times I),\dots, \Psi_i(a_k\times I)\}$ gives a cut system of disks for $H_i$. We want to see the boundaries of these disks on our fixed identification of $\Sigma$. Namely, we want to describe the curves $\Psi_i^{-1}(a_j\times\{0\} \cup a_j\times\{1\})$ in $F_0\cup F_1$. Each $\Psi_i$ is the identity on $F_0$, and defines parallel transport along the arc from $(0,1)$ to $(0,-1)$ over which $H_i$ lies. Since the $\Psi_i$ define parallel transport, and we are using $\Psi_1$ to identify $\Sigma$ with $F_0\cup F_1$, we see that $\Psi_i(a_j\times\{1\})$ in $\Sigma$ is the image of $a_j$ under the monodromy around the curve which goes from $(0,-1)$ to $(0,1)$ along the $H_i$ curve and then goes from $(0,1)$ to $(0,-1)$ along the $H_1$ curve. This monodromy is $\tau(c_1),..., \tau(c_{i-1})$. Thus, the cut system for $H_i$ is obtained from the cut system for $H_{i-1}$ by applying the right handed Dehn twist $\tau(c_{i-1})$ where $c_{i-1}\subset F_1$. Any choice of complete arc system defines a cut system for $H_1$ by gluing together the same arcs on $F_0$ and $F_1$.

To summarize, given a PALF with fiber $F$ and ordered vanishing cycles $\{c_1,\dots, c_n\}$, the corresponding multisection with divides is given by

\begin{itemize}
    \item $\Sigma = F_0\cup_\partial F_1$ where $F_0$ and $F_1$ are copies of $F$ ($F_0$ is oppositely oriented).
    \item The dividing set $d$ is $\partial F = F_0\cap F_1$.
    \item The cut system for $H_1$ is $\{ a_1\cup a_1, \dots, a_g\cup a_g\}$ where $\{a_1,\dots, a_g\}$ is a complete arc system for $F$.
    \item For $i>1$, the cut system for $H_i$ is obtained from the cut system for $H_{i-1}$ by applying a right handed Dehn twist about $c_{i-1}\subset F_1$.
\end{itemize}

Note that each cut curve intersects the dividing set in exactly two points (the end points of the arc $a_j$). (This is clearly true for $H_1$, and it remains true for each $H_i$ because the Dehn twists are applied in the interior of $F_1$ which is disjoint from the dividing set.)

\begin{figure}
    \centering
    \includegraphics[scale =.35]{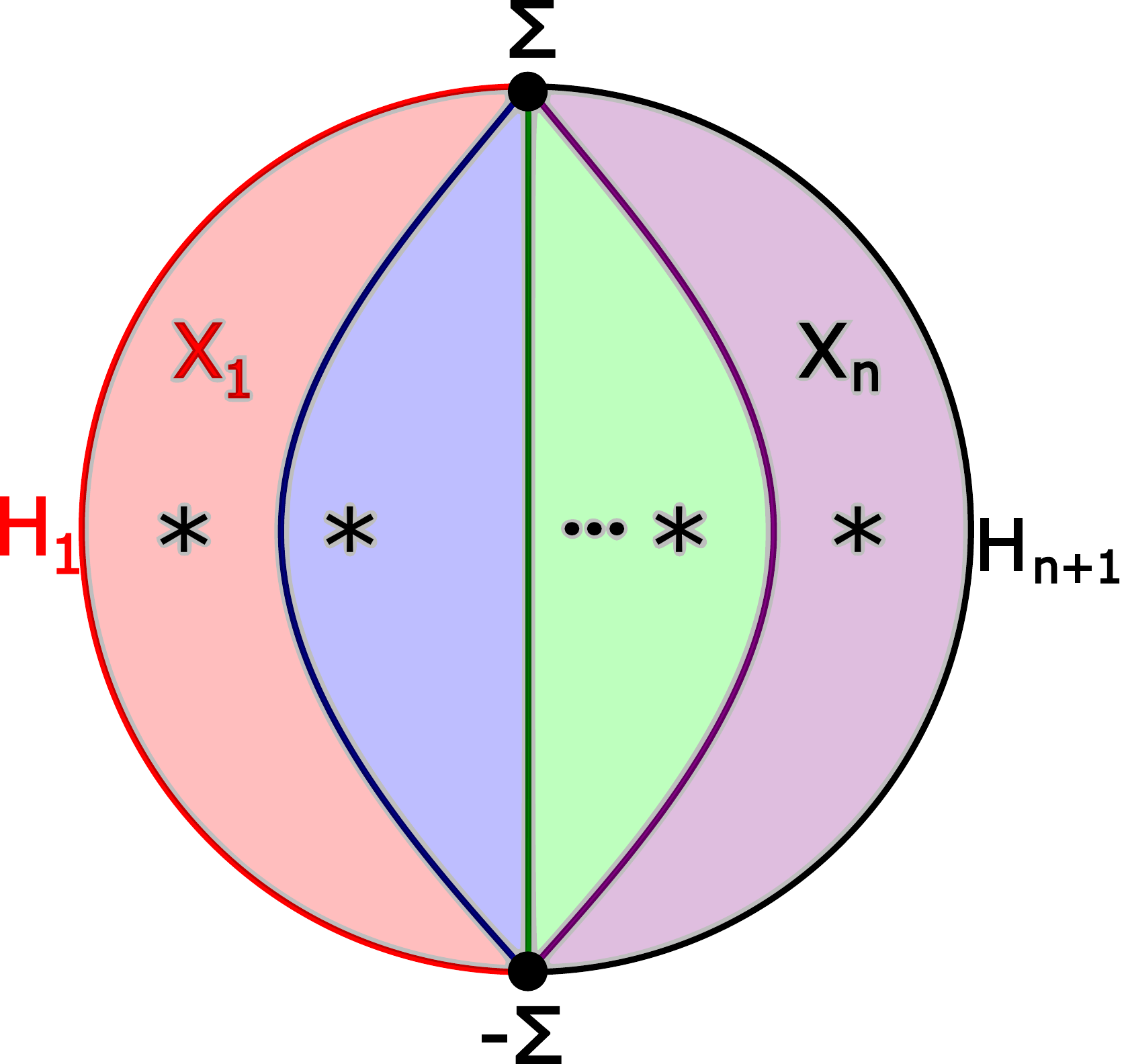}
    \caption{Decomposing a PALF into pieces containing exactly one Lefschetz singularity yields a multisection. The multisection surface is two fiber surfaces glued along their binding and each handlebody is a product region between these two surfaces.}
    \label{fig:palfToMultisection}
\end{figure}

\begin{remark} \label{rem:PALFstd}
    Note that the output of a PALF yields a slightly more general condition which ensures that a multisection with divides corresponds to a Weinstein manifold. In this case, every consecutive pair of cut systems differs by a right handed Dehn twist about a curve which lies entirely in the $\Sigma_+$ side of the dividing set. Let's call such $(C_i,C_{i+1},d)$ \emph{generalized standard Weinstein cobordant} (gsWc). (This generalizes the notion of standard Weinstein cobordant from Proposition~\ref{prop:standardDiagramsAreWeinstein}.) If consecutive cut systems $C_i$ and $C_{i+1}$ are gsWc where the curve defining the Dehn twist relating them is dual to one of the components of $C_i$, they define a smooth multisection by Lemma~\ref{lem:twistingOnDualCurvesGivesMultisection}. Furthermore, whenever we have gsWc cut systems representing handlebodies $H_i$ and $H_{i+1}$ forming the Heegaard splitting on the boundary of a sector $W_i$, we can again interpret $W_i$ as a Weinstein cobordism from $H_i$ to $H_{i+1}$. This is because we can use the Lefschetz fibration interpretation of $W_i$ and then view the Lefschetz critical point as an attachment of a Weinstein $2$-handle to $H_i\times I$. Since $H_i\times I$ can either be interpreted as a filling or a Weinstein cobordism, the sector $W_i$ also supports a Weinstein structure making it a filling of its boundary, and a Weinstein structure making it a cobordism from $H_i$ to $H_{i+1}$. This shows that Proposition~\ref{prop:standardDiagramsAreWeinstein} holds under the more general assumption that consecutive pairs of cut systems are gsWc, and the curve defining the Dehn twist for a pair of cut systems is dual to one of the curves from the first cut system.
\end{remark}

The previous construction suggests that multisection can also be encoded by a monodromy factorization, and here, we show that this is indeed the case. Through some care, one can determine a particular monodromy factorization from a multisection, however for the present paper it will be sufficient to have a family of possible factorizations.

Fixing a surface $\Sigma$, a system of cut curves defines a handlebody with boundary $\Sigma$. Given a handlebody, we can choose any system of cut curves which bound disks in the handlebody to encode it. If we specify a diffeomorphism $\Phi$ of $\Sigma$, we can apply that diffeomorphism to a cut curve system, to produce a new cut curve system. If the original cut curve system defines a handlebody $H$, we let $H_\Phi$ denote the handlebody defined by the image of a cut-system for $H$ under the map $\Phi$. Note this is independent of the choice of cut-system for $H$, because two different cut systems $\{\alpha_i\}$ and $\{\alpha_i'\}$ for $H$ will be related by some sequence of handle slides, and thus $\{\Phi(\alpha_i)\}$ will be related to $\{\Phi(\alpha_i')\}$ by a sequence of handle slides as well. We begin with a lemma which yields a sufficient condition for $H$ and $H_\Phi$ to form a Heegaard splitting of $\#_{g-1} S^1 \times S^2$ when $\Phi$ is a Dehn twist.

\begin{lemma}
\label{lem:twistingOnDualCurvesGivesMultisection}
    Let $\Sigma$ be a closed genus $g$ surface, $c$ be a simple closed curve on $\Sigma$, and $\tau_c$ be a (right or left handed) Dehn twist about $c$. Let $H$ be a handlebody with boundary $\Sigma$. Suppose there exists a properly embedded disk $D \subset H$ whose boundary, $\partial D = E_1$, is non-separating on $\Sigma$ such that $|E_1 \cap c| = 1$. Then $H \cup_\Sigma H_{\tau_c}$ is a Heegaard splitting of $\#_{g-1} S^1 \times S^2$.
\end{lemma}

\begin{proof}
    We will produce a Heegaard diagram of $H \cup_\Sigma H_{\tau_c}$ which consists of $g-1$ parallel curves and two curves intersecting once, which proves the lemma. Extend $E_1$ to a cut system, $E$, for $H$. Now each point of $c \cap (E \backslash \{E_1\})$ can be eliminated by the following process: start at the point $c \cap E_1$ and follow $c$ until we meet the first point of $c \cap (E \backslash \{E_1\})$ call the resulting arc $a$ and the curve in $E$ at this intersection $E_i$. Sliding $E_i$ over $E_1$ along $a$ eliminates this intersection and does not introduce any new intersections between $E$ and $c$. Repeating this process, using narrower and narrower sliding bands, we may eliminate all intersections of $c$ and $E$ except for the point $c \cap E_1$.

    We may then obtain a Heegaard diagram for $H \cup_\Sigma H_{\tau_c}$ by taking the Dehn twist of the cut system resulting from these slides on $E$ about $c$. As $c$ is disjoint from all of the curves other than $E_1$ there are $g-1$ curves which are unchanged by the twisting and are therefore parallel. On the other hand $E_1$ and $ \tau_c(E_1)$ intersect once, providing the desired Heegaard diagram.
\end{proof}

Using the above lemma, we obtain a sufficient criteria for Dehn twists on cut systems to yield two sequential handlebodies in a multisection. Conversely, the following proposition shows that the sequential handlebodies in all multisections can be obtained in the fashion.

\begin{proposition}
    Let $\mathfrak{M} = X_1 \cup X_2 \cdots \cup X_n$ be a genus-$g$ multisection with multisection surface $\Sigma$, and with $X_i \cong \natural_{k_i} S^1 \times D^3$. Let $H^{1}, H^{2},  \dots , H^{n}, H^{n+1}$ be the 3-dimensional handlebodies lying at the boundaries of the $X_i$. Then there exist curves $c_1^1,c_2^1, \dots ,c_{g-k_1}^{1}, c_{1}^2,c_{2}^2, \dots ,c_{g-k_2}^{2}, \dots , c_{1}^{n}, c_{2}^{n}, \dots, c_{g-k_n}^{n}$ such that $H^{i}_{\tau(c_{1}^i)\circ \tau(c_{2}^i) \dots \cdot \circ \tau(c_{g-k_i}^{i})} = H^{i+1}$. 
\end{proposition}

\begin{proof}
    For each handlebody, $H^i$, we will show how to produce the curves $c_i^1 ... c_i^{g-k_i}$.  Recall that the handlebodies $H_i$ and $H_{i+1}$ meet at the multisection surface $\Sigma$ to form a a Heegaard splitting of $\#^{k_i} S^1 \times S^2$. Consider a Heegaard diagram of $H_i \cup H_{i+1}$. By Waldhausen's theorem \cite{Wald68}, after a sequence of handle slides there is a cut system of curves $a_1,$ ... $,a_g$ for $H_i$ and $b_1 ... b_g$ for $H_{i+1}$ such that $a_i = b_i$ for  $0 \leq i \leq k_i$ and $|a_n \cap b_m| = \delta_{n,m}$ for  for $k_{i+1} \leq n,m \leq g$. 
    
    For $k_i \leq j \leq g$ we let $c_i^j = \tau_{b_j}(a_j)$. Then, $\tau_{c_i^j}(a_j) = b_j$. Moreover, since $c_j^i$ does not intersect any of the other $a_k$ for $k \neq j$, $\tau_{c_{j}^i}(a_k) = a_k$ for $k \neq j$. Then the product $\Pi_{j=1}^{k_i} \tau_{c_{j}^i}(a_k)$ takes a cut system for $H_i$ to a cut system for $H_{i+1}$. 
\end{proof}

We call the product $\Pi_{i=1}^n \Pi_{j=1}^{g-k_i} \tau(c_{j}^i)$ a \textbf{monodromy factorization} for $\mathfrak{M}$. By following through our construction in Theorem \ref{thm:palfToMultisection}, we can track the monodromy of a PALF onto the monodromy of a multisection which immediately yields the following.

\begin{corollary}
\label{cor:factorizationOfDouble}
    Suppose that $(X^4, \omega)$ is a Weinstein manifold which admits a PALF with fiber surface $F$ and monodromy factorization $P = \Pi_{i=1}^n \tau (c_i)$. Then $(X^4, \omega)$ admits an n-section with divides with multisection surface $\Sigma = F \cup -F$, dividing set $\partial F$, and monodromy factorization $P'$ obtained by applying the Dehn twists of $P$ to $F \subset \Sigma$.
\end{corollary}

\subsection{Monodromy Substitution}
\label{subsec:MonodromySubstitution}
In this section we will demonstrate how a monodromy substitution affects a multisection with divides. We begin with the analogous construction for PALFs.

\begin{definition}
Let $f:M^4 \to D^2$ be a PALF with fiber surface $\Sigma$ and monodromy factorization $\Pi_{i=1}^n \tau(c_i)$. Suppose that for some $k,l,m,n$ and curves $c_m'...c_n'$ we have that, as mapping classes, $\Pi_{i=k}^l \tau(c_k) = \Pi_{j=m}^n \tau(c_j')$. Then we may obtain a new Lefschetz fibration with monodromy factorization given by 
\[\Pi_{i=1}^k \tau(c_i)\Pi_{j=m}^n \tau(c_j') \Pi_{i=l+1}^n \tau(c_i).\] We say that the new Lefschetz fibration is obtained by a \textbf{monodromy substitution} on $f$.
\end{definition}

Monodromy substitution has been used extensively to produce new symplectic manifolds from existing ones. In particular, in \cite{EndGur10}, the authors show that the lantern relation, pictured in Figure \ref{fig:lanternRelation}, can be used to perform a rational blowdown on the configuration $C_2$ (see \cite{GomSti99} Section 8.5 for an exposition on these operations). This was later generalized in \cite{EMV11} to realize an infinite family of rational blowdowns as monodromy substitutions using daisy relations. In general, any monodromy substitution can be thought of as some symplectic cut-and-paste operation.

\begin{figure}
    \centering
    \includegraphics[scale=.3]{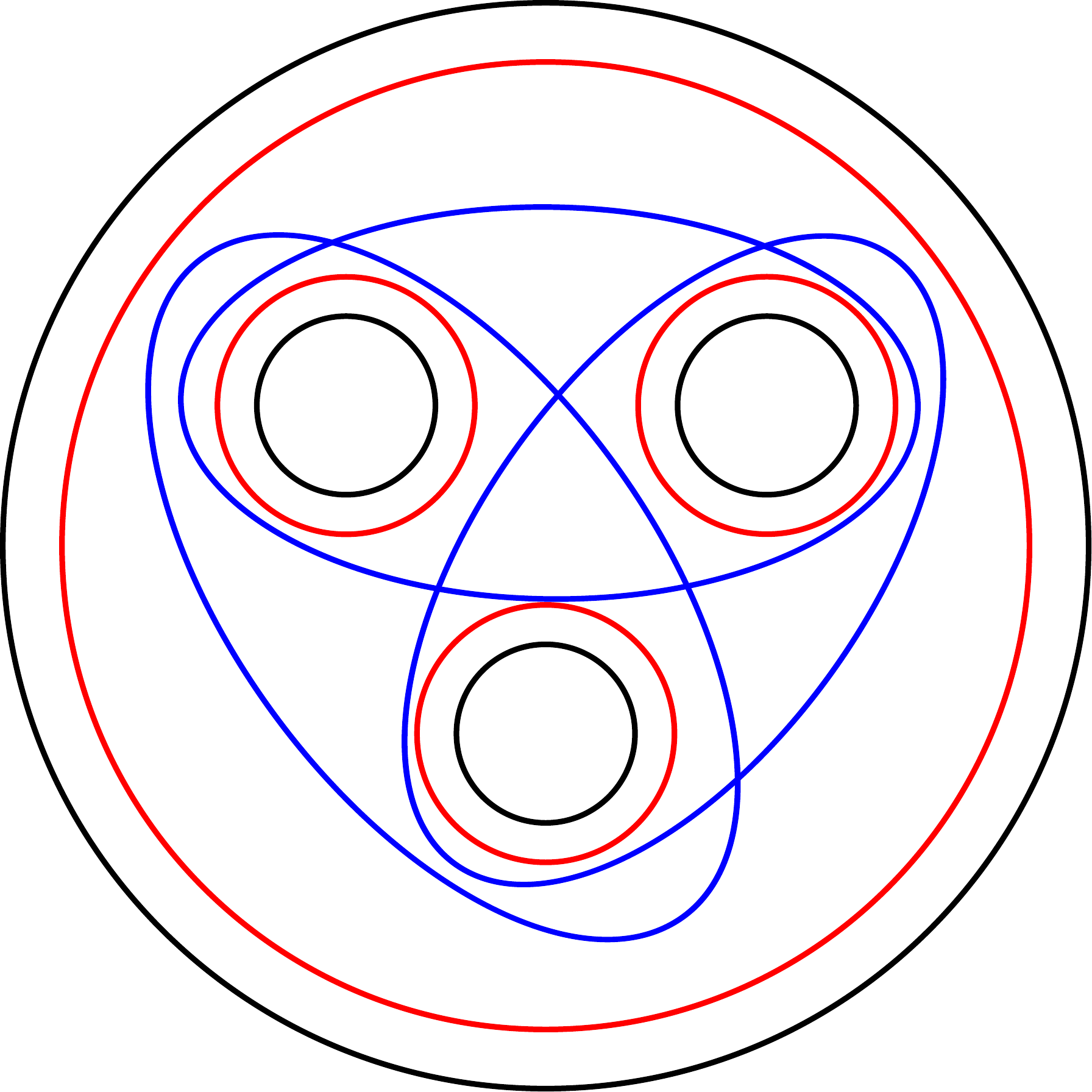}
    \caption{The lantern relation in the mapping class group in a 4-holed sphere states that right handed Dehn twists about the red curves gives the same mapping class as the right handed Dehn twists about the blue curves.}
    \label{fig:lanternRelation}
\end{figure}

There is an analogous process of monodromy substitution on a multisection.

\begin{definition}
    Let $\mathfrak{M}$ be a multisection (not necessarily with boundary) starting at the handlebody $H$ with monodromy factorization given by $\Pi_{i=1}^n \tau(c_i)$. Let $H^k$ be the handlebody $H_{\tau(c_1)\cdot\tau(c_2)\cdot \dots \cdot \tau(c_k)}$ and suppose that $c_{k+1}', c_{k+2}',\dots, c_{j'}'$ is a sequence of curves such that for all $l \in \{k+1, k+2,\dots, j'\}$ we have that $c_l'$ is dual to some disk in $H^k_{\tau(c_{k+1}')\cdot \dots \cdot \tau(c_{l-1}')}$ (this will guaranteed that the assumptions of Lemma \ref{lem:twistingOnDualCurvesGivesMultisection} hold).  Suppose further that $H^k_{\tau(c_{k+1})\cdot \tau(c_{k+2})\cdot \dots \cdot \tau(c_{j})} = H^k_{\tau(c_{k+1}')\cdot\tau(c_{k+2}')\cdot \dots \tau(c_{j'}')}$. Then we may obtain a new multisection $\mathfrak{M}_S$ starting at the handlebody $H$ and specified by the monodromy factorization $\Pi_{i=1}^k \tau(c_i) \Pi_{i=k+1}^{j'} \tau(c_i') \Pi_{i=j}^n \tau(c_i)$. We call $\mathfrak{M}_s$ a \textbf{monodromy substitution} of $\mathfrak{M}$.
\end{definition}

It follows immediately from Corollary \ref{cor:factorizationOfDouble} that we can find monodromy substitutions by doubling a PALF and a monodromy substitution of that PALF. Carrying this out for the lantern relation gives us a monodromy substitution on a multisection with divides yielding the operation outlined in Figure \ref{fig:lanternRelationOnBisections}.

\begin{figure}
    \centering
    \includegraphics[scale=.25]{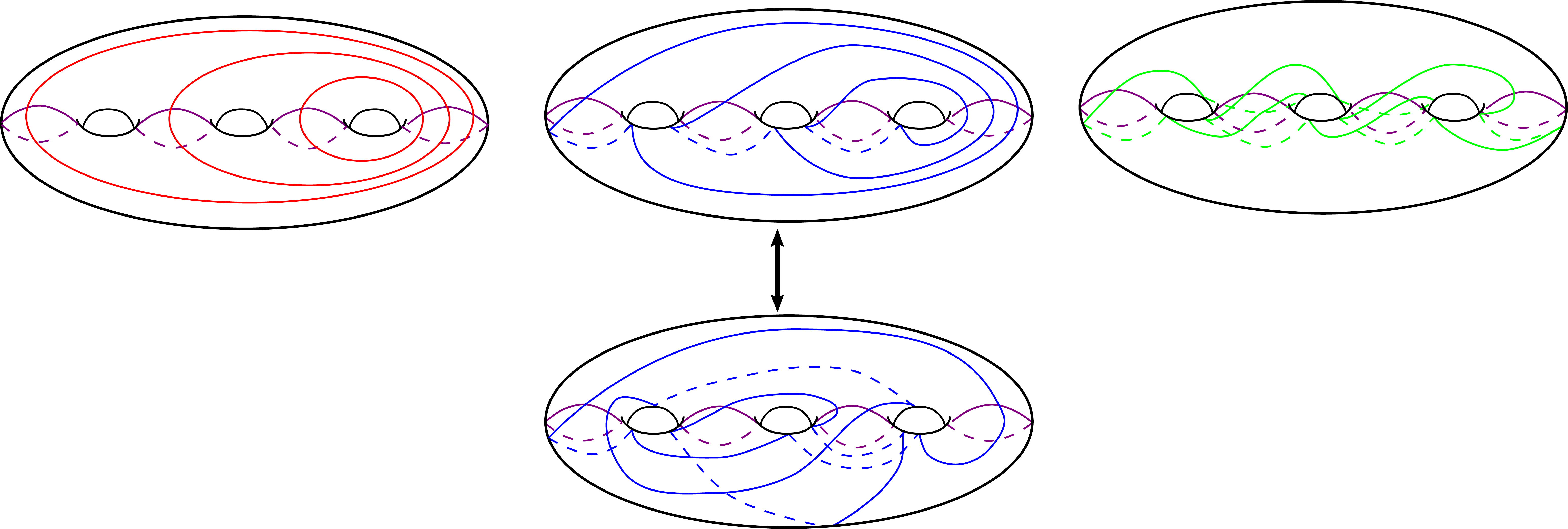}
    \caption{The three handlebodies on the top row yield a bisection whose monodromy is the double of the right handed Dehn twists about the red curves in Figure \ref{fig:lanternRelation}. Replacing the middle handlebody by the one below it yields a multisection whose monodromy is the double of the right handed Dehn twists about the blue curves in Figure \ref{fig:lanternRelation}. The overall change in a bisection containing these handlebodies is a $C_2$  rational blowdown. }
    \label{fig:lanternRelationOnBisections}
\end{figure}

\section{Genus-1 multisections and stabilizations} \label{s:g1stab}

\subsection{Classification of genus 1 multisections with divides}
\label{subsec:genus1Classification}

In this subsection we will provide a characterization of genus-1 multisections with divides. For examples of the diagrams for the unique 2- and 3-sections with divides, see Figure~\ref{fig:genusOne}. Smooth genus-1 multisections are well characterized by their diagrams, which consist of sequences $(\alpha_1, ... \alpha_n)$ with $|\alpha_i \cap \alpha_{i+1}| = 1.$ In \cite{IKLM}, the authors show that smooth genus-1 $n$-sections with boundary correspond to linear plumbings of $(n-1)$ disk bundles over the sphere. Moreover, given the oriented sequence of cut curves $(\alpha_1, \alpha_2, ... \alpha_n)$ defining a genus-1 multisection diagram, the Euler number of the $i^{th}$ disk bundle is given by the algebraic intersection $\langle \alpha_{i-2}, \alpha_i \rangle$.

\begin{prop}
\label{prop:Genus1Classification}
There is a unique genus-1 n-section with divides for each $n \geq 2$. These correspond to the linear plumbing of $(n-1)$ disk bundles of Euler number $-2$ over the sphere ($T^*S^2$'s).
\end{prop}

\begin{proof}

First we observe that the linear plumbing of $(n-1)$ disk bundles of Euler number $-2$ supports a PALF structure whose fiber is an annulus with $n$ vanishing cycles all parallel to the core circle of the annulus. By the algorithm in Theorem~\ref{thm:palfToMultisection}, these Weinstein domains have genus $1$ multisections.

Now we show that these are the only genus $1$ multisections with divides.

For a contact Heegaard splitting of $S^3$, the dividing set consists of two parallel curves. Fixing coordinates $(a,b)$ for $H_1(T^2)$, we may assume, after an orientation preserving homeomorphism that $\alpha_1 = (0,1)$ and $\alpha_2 = (1,0)$. The dividing set will be two parallel curves of slope $d$. Note that $\alpha_i$ intersects the dividing set twice if and only if $|\alpha_i \cap d| = 1$. Therefore $d=(1,\pm 1)$. Since $d=(1,1)$ corresponds to a contact Heegaard splitting for an overtwisted contact structure on $S^3$, we must have $d=(1,-1)$.

We first treat the case $n=2$, and then proceed inductively. In this case, we seek to find the possible slopes for $\alpha_3$. As $|\alpha_2 \cap d| = 1$, all curves which intersect $\alpha_2$ once are given by Dehn twists of $\alpha_2$ about $d$. In addition the requirement that $|\alpha_3 \cap d| = 1$ means that $\alpha_3$ is a single Dehn twist of $\alpha_2$ about $d$. If this Dehn twist is left handed, then the quadruple $(\Sigma, d, \alpha_2, \alpha_3)$ gives a diagram for the overtwisted $S^3$, so the Dehn twist must be right handed. Therefore $\alpha_3 = (1,-2)$ and by the classification of smooth genus 1 multisections $(\Sigma,\alpha_1, \alpha_2,\alpha_3, d)$ gives a bisection with divides of the disk bundle of Euler number $-2$ over the sphere.

In general suppose that $(\alpha_1 ... \alpha_{n-1})$ is a sequence of curves defining a $(n-1)$-section with divides. Then, as in the base case, $\alpha_{n-1}$ is a right handed Dehn twist of $\alpha_{n-2}$ about $d$ and $\alpha_n$ is a right handed Dehn twist of $\alpha_{n-1}$ about $d$. Therefore, $\langle \alpha_{n-2}, \alpha_n \rangle = -2$ so that we have indeed plumbed an additional $-2$-sphere.
\end{proof}

\begin{figure}
    \centering
    \includegraphics[scale=.3]{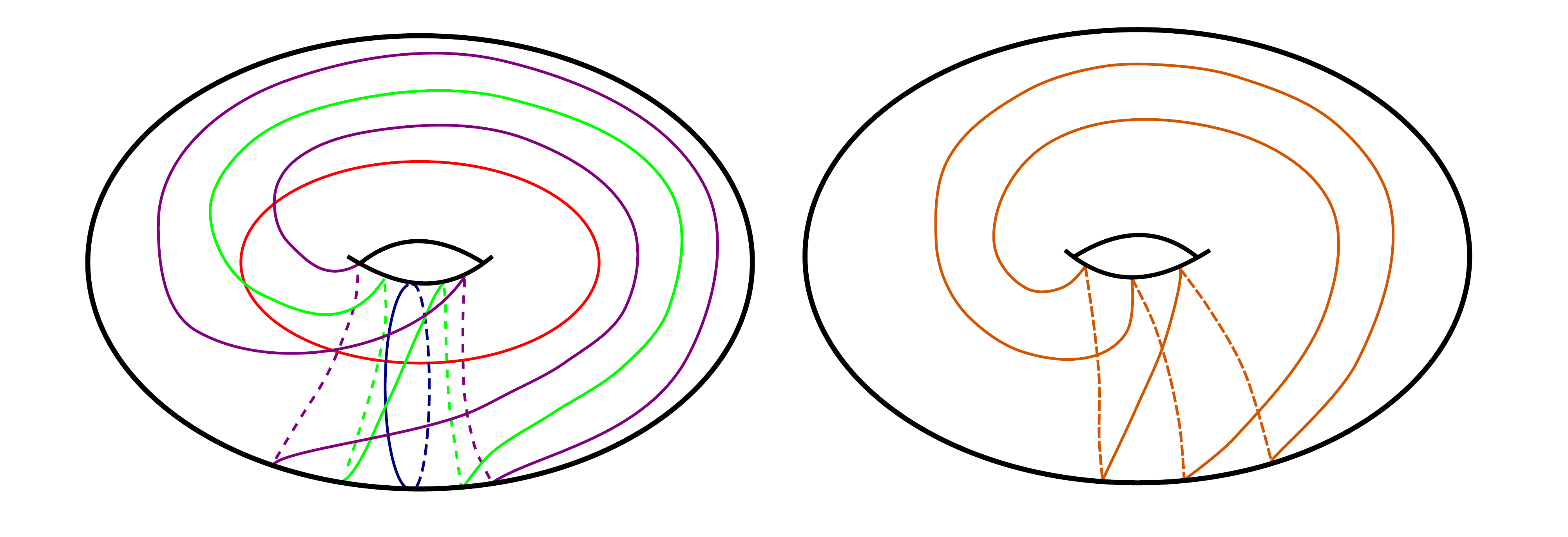}
    \caption{Left: The unique genus-1 bisection with divides corresponding to the disk bundle over the sphere of Euler number $-2$. Right: Adding this curve to the decomposition yields the unique genus-1 3-section with divides.}
    \label{fig:genusOne}
\end{figure}

\subsection{Stabilization}
\label{subsec:Stabilization}

Here, we will introduce an operation on multisections with divides which takes a genus-$g$ $n$-section and produces a genus-$(g+1)$ $(n+1)$-section. An explicit example of this process applied to the genus-1 bisection of $T^*S^2$ can be seen in Figure \ref{fig:stabilizingTStarS2}. 

This stabilization operation can be seen from both the perspective of a handle decomposition, as in Section \ref{ss:existenceKirby} or from the perspective of a PALF. We will focus on the second perspective, and we recall the definition of a stabilization of a PALF. 

\begin{definition}
    Let $f: M^4 \to D^2$ be a PALF with fiber surface $\Sigma$ 
    monodromy factorization $\Pi_{i=1}^n \tau(c_i).$ Then a \textbf{stabilization} of $f$ is a PALF with fiber surface $\Sigma'$ and monodromy factorization $\Pi_{i=1}^{n+1} \tau(c_i)$ where $\Sigma'$ is obtained by attaching a 2-dimensional 1-handle to $\Sigma$ and $c_{n+1}$ is a curve on $\Sigma'$ intersecting the belt sphere of the attached 1-handle geometrically once.
\end{definition}

\begin{figure}
    \centering
    \includegraphics[scale=.4]{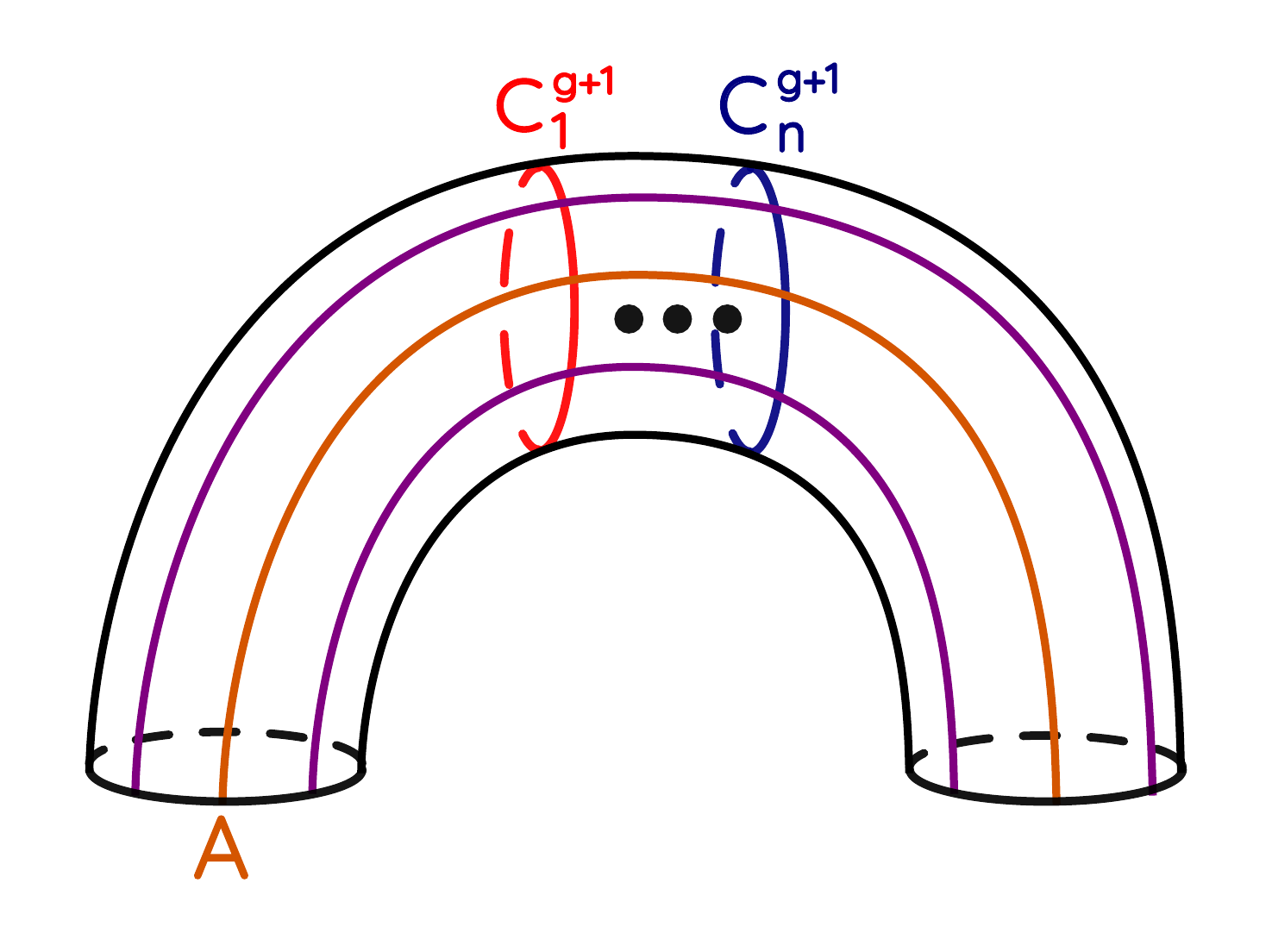}
    \caption{The annulus used to perform a stabilization of a multisection with divides. The existing cut systems $C_1...C_n$ each receive a new curve $c_i^{g+1}$ to form cut systems $C_i'$ for $i \in \{1,...,n \}$. The arc $A$ glues with an arc in the existing multisection which is disjoint from the dividing set to yield a curve $C$. Performing a Dehn twist about $C$ to each curve in the cut system $C_n'$ yields a new handlebody $C_{n+1}'$ so that the sequence of cut systems $(C_1', ..., C_{n+1}')$ is a new multisection diagram.}
    \label{fig:stabilizationDiagram}
\end{figure}

By doubling the 1-handle used in a stabilization of a PALF, we obtain the diagram in Figure \ref{fig:stabilizationDiagram}. When glued to a multisection with divides appropriately we obtain a new multisection with divides, whose construction is outlined in the following definition.

\begin{definition}
    Let $\mathfrak{M}$ be a multisection with divides with a diagram given by $(\Sigma, d, C_1, ..., C_n)$. Suppose $C_i$ is made up of the curves $(c_i^1, ..., c_i^g)$. Let $p_1$ and $p_2$ be two points on $d$ and $A'$ be an arc between $p_1$ and $p_2$ which is contained in $\Sigma^+$. Let $\Sigma'$ be the surface obtained by removing neighbourhoods of $p_1$ and $p_2$ and gluing in the surface shown in Figure \ref{fig:stabilizationDiagram} to the resulting boundary where the gluing is performed so that the arc $A$ meets the arc $A'$ to form a curve $c$. By slight abuse of notation, we will keep the notation $c_i^k$ for the simple closed curve on $\Sigma'$ after performing surgery. We then obtain a new multisection $\mathfrak{M}_S$ with diagram given by $(\Sigma', d, C_1', ..., C_n', C_{n+1}')$ where, for i $\in \{1, ... , n \}$, $C_i' = (c_i^1, ..., c_i^g , c_i^{g+1})$ and $C_{n+1}' = (\tau_c( c_n^1 ), ..., \tau_c( c_n^g ) , \tau_c( c_n^{g+1} ))$.
\end{definition}

We can easily see that $\mathfrak{M}_S$ is still a multisection diagram with divides. The condition that each curve in each cut system intersects the dividing set in two points is ensured by looking at the model in Figure~\ref{fig:stabilizationDiagram} and the fact that the curve we Dehn twist about is disjoint from the dividing set. That $\mathfrak{M}_S$ still represents a multisection smoothly follows from Lemma~\ref{lem:twistingOnDualCurvesGivesMultisection}. That $(C_n',C_{n+1}',d)$ represents a contact Heegaard splitting of the \emph{tight} contact structure on $\#_{g-1} S^1\times S^2$ follows from the fact that we can obtain this contact Heegaard splitting as the boundary of a PALF sector as in the proof of Theorem~\ref{thm:palfToMultisection}.

Observe that the manifold represented by $(\Sigma',d,C_1',\dots, C_n')$ is related to the manifold represented by $(\Sigma,d,C_1,\dots, C_n)$ by attaching a single $1$-handle. Thus if $(\Sigma,d,C_1,\dots, C_n)$ represents a Weinstein domain, then $(\Sigma',d,C_1',\dots, C_n')$ also represents a Weinstein domain obtained by attaching a single Weinstein $1$-handle. (Note that there is no constraint on the attaching data for a $1$-handle attachment to be Weinstein.) By Remark~\ref{rem:PALFstd}, the new sector $W_n$ amounts to attaching a Weinstein cobordism to $W''=W_1'\cup\dots \cup W_{n-1}'$. Therefore, the stabilized diagram $(\Sigma',d,C_1',\dots, C_n',C_{n+1}')$ also represents a Weinstein domain. Furthermore, the Weinstein cobordism from the sector $W_n$ attaches a Weinstein $2$-handle which cancels with the added $1$-handle (the attaching sphere of the $2$-handle intersects the belt sphere of the $1$-handle in one point). Therefore, in total, we have added a trivial Weinstein cobordism (one which is Weinstein homotopic to a trivial cobordism). Namely, if before stabilization, the multisection diagram with divides represented a Weinstein domain, then after stabilization, the multisection diagram with divides represents the same Weinstein domain up to Weinstein homotopy. We summarize in the following proposition.

\begin{proposition}
    Let $\mathfrak{M}$ be a multisection with divides which admits a multisection diagram with divides given by $(\Sigma, d, C_1, ..., C_n)$. Let $\mathfrak{M}'$ be a stabilization of $\mathfrak{M}$ with a multisection diagram for $\mathfrak{M}'$ given by $(\Sigma', d,  C_1', ..., C_{n+1}')$. Then, $\mathfrak{M}'$ is a multisection diagram with divides, and the Weinstein manifolds encoded by $\mathfrak{M}$ and $\mathfrak{M}'$ are Weinstein homotopic.
\end{proposition}

\begin{figure}
    \centering
    \includegraphics[scale=.32]{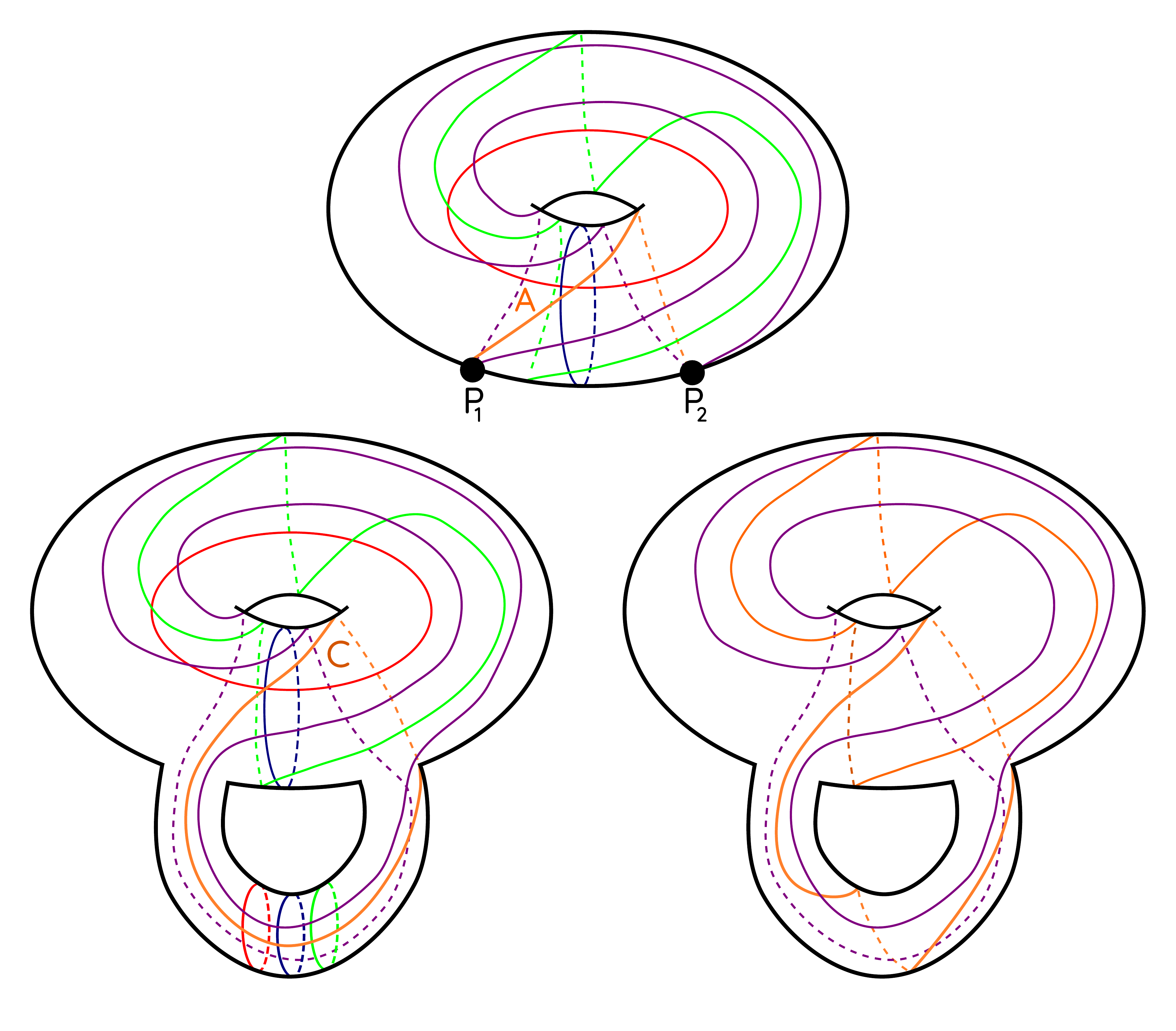}
    \caption{This figure gives the process for stabilizing the genus-1 multisection diagram  for $T^*S^2$. Top: The genus-1 multisection diagram for $T^*S^2$ together with two points on the dividing set and an arc $A$ whose interior is disjoint from the dividing set. Bottom left: Gluing the stabilizing annulus in Figure \ref{fig:stabilizationDiagram} yields the first three cut systems for the stabilization, together with a curve $C$. Bottom Right: Dehn twisting the green cut system about the curve $C$ yields the final cut system for the stabilized multisection.}
    \label{fig:stabilizingTStarS2}
\end{figure}

\section{Questions} \label{s:questions}

Donaldson \cite{Don96} and Giroux \cite{Gir02,Gir17} proved that every symplectic manifold admits a symplectic divisor such that the complement of a standard neighborhood of the divisor is a Weinstein domain. In search of a diagrammatic theory for closed symplectic manifolds, one strategy would be to find a suitable structure on the neighborhood of the divisor,and glue as in \cite{IslNay20} to a multisection with divides for the Weinstein complement. This leads us to the following questions.

\question{Can we construct analogous multisections with divides for symplectic $4$-manifolds with \emph{concave} boundary? In particular, concave neighborhoods of symplectic divisors. Do we need different diagrammatic information to encode concave boundary? How do we specify in a multisection diagram how to symplectically glue convex pieces to concave pieces?}

The results of Section \ref{s:PALF} primarily consisted of using PALFs to obtain information about multisections with divides, but in favorable conditions (see Remark \ref{rem:PALFstd}), this construction can be reversed to obtain a PALF from a multisection with divides. It is an open question as to whether two PALFs corresponding to the same Weinstein 4-manifold are related by stabilization, Hurwitz equivalence, and an overall conjugation. We have translated the stabilization move in Section \ref{subsec:Stabilization}, and using a similar approach, the other two moves can readily be translated into moves on multisections with divides. Here, techniques used in the stable equivalence of trisections in \cite{GayKir16} could prove fruitful in addressing the following question.

\question{By passing to the related multisection with divides, can we show that any two PALFS corresponding to the same Weinstein 4-manifold are related by stabilization, Hurwitz equivalence, and an overall conjugation?}

A preliminary result one would need in order to answer the previous question would be a uniqueness result for multisections with divides. 

\question{Let $W_1$ and $W_2$ be multisections with divides for the same underlying Weinstein manifold $W$. What is a sufficient set of moves relating $W_1$ and $W_2$.}

In \cite{Gay19} and \cite{Isl21} the authors show that every multisection with multisection surface $\Sigma_g$ can be realized as a generic path of smooth real-valued functions on $\Sigma_g$. Much of the data of these functions is discarded in this process, so that, in the end, the smooth topology is determined by the level sets of the regular times. By keeping track of more information, it is likely that one provide an answer to the following question which would link symplectic topology with the theory of smooth functions on surfaces.

\question{Which generic paths of smooth functions $f: \Sigma \times [0,1] \to \mathbb{R}$ yield multisections with divides?}

In Proposition \ref{prop:standardDiagramsAreWeinstein} and Remark~\ref{rem:PALFstd}, we gave a sufficient condition for a multisection diagram with divides to correspond to a Weinstein manifold. This condition is likely not necessary and is, in practice, difficult to check. We therefore pose the following question.

\question{Which multisection diagrams with divides correspond to Weinstein manifolds?}

One of the conditions we require for a multisection with divides is that the contact structure on the boundary of each sector is tight. While checking tightness is generally quite difficult, we are dealing with the restricted class of connected sums of $S^1 \times S^2$ where the problem is likely more tractable. In particular, in this situation, the non-vanishing of contact invariant in Heegaard-Floer homology is equivalent to tightness \cite{OzsSza05}. Moreover the Heegaard-Floer homology of this manifold is particularly simple. Since each sector in a multisection has a contact Heegaard splitting  on its boundary, we posit the following question. 

\question{Is there a simple algorithm for computing the contact invariant of a contact structure on $\#_k S^1 \times S^2$ from a contact Heegaard diagram?}

By analyzing a PALF filling of an open book supporting a contact structure Oszv\'ath and Szab\'o show that the contact invariant invariant is non-vanishing for Weinstein fillable manifolds~\cite{OzsSza05}. A multisection with divides seems to encode similar information to a PALF, but with Heegaard splittings playing the role of open book decompositions. In light of this, it is possible that one can give a more direct proof of this non-vanishing result. 

\question{Given a multisection with divides, locate the contact element of the three manifold given by the contact Heegaard splitting of the boundary. Show that if the multisection consists of standard Weinstein cobordisms (as in Proposition \ref{prop:standardDiagramsAreWeinstein} or the generalization in Remark~\ref{rem:PALFstd}) this element is non-vanishing.}

\bibliographystyle{amsalpha}
\bibliography{thebib.bib}
\end{document}